\documentclass[12pt,letterpaper]{article}
\usepackage{amsmath,amsfonts,amsthm,amssymb}

\swapnumbers 

\newtheorem{lemma}{Lemma}[section]
\newtheorem{corollary}[lemma]{Corollary}
\newtheorem{theorem}[lemma]{Theorem}
\newtheorem{example}[lemma]{Example}

\theoremstyle{definition} 
\newtheorem{definition}[lemma]{Definition}
\newtheorem{remark}[lemma]{Remark}

\newtheorem{proposition}[lemma]{Proposition}

\newcommand\integers{{\mathbb Z}}

\begin{document}

\title{Unitizations of generalized pseudo effect algebras and their ideals}

\author{David J. Foulis,{\footnote{Emeritus Professor, Department of
Mathematics and Statistics, University of Massachusetts, Amherst,
MA; Postal Address: 1 Sutton Court, Amherst, MA 01002, USA;
foulis@math.umass.edu.}}\hspace{.05 in} Sylvia
Pulmannov\'{a} and Elena Vincekov\'{a}{\footnote{ Mathematical Institute,
Slovak Academy of Sciences, \v Stef\'anikova 49, SK-814 73 Bratislava,
Slovakia; pulmann@mat.savba.sk, vincek@mat.savba.sk. \newline The second
and third authors were supported by Research and Development Support Agency
under the contract No. APVV-0178-11 and grant VEGA 2/0059/12.}}}

\date{}

\maketitle

\begin{abstract}
\noindent A generalized pseudo effect algebra (GPEA) is a partially
ordered partial algebraic structure with a smallest element $0$, but
not necessarily with a unit (i.e, a largest element). If a GPEA admits
a so-called unitizing automorphism, then it can be embedded as an
order ideal in its so-called unitization, which does have a unit.
We study unitizations of GPEAs with respect to a unitizing
automorphism, paying special attention to the behavior of
congruences and ideals in this setting.
\end{abstract}

\section{Introduction}

An \emph{effect algebra} (EA) is a bounded partially ordered
structure equipped with a partially defined commutative and
associative binary operation called the \emph{orthosummation},
often denoted by $\oplus$ \cite{FB}. The smallest element in an
EA, called the \emph{zero} or \emph{neutral} element, is usually
denoted by $0$, and the largest element, called the \emph{unit},
is often denoted by $1$. For each element $a$ in an EA, there is a
unique element $a\sp{\perp}$, called the \emph{orthosupplement} of
$a$, such that $a\oplus a\sp{\perp}=1$.

Effect algebras were originally introduced to represent possibly
fuzzy or unsharp propositions arising either classically \cite{SPGfuzzy}
or in quantum measurement theory \cite{BLM}.

The class of EAs includes the class of orthoalgebras \cite{FGR} and
the class of orthomodular posets \cite[p. 27]{Kalm}. The subclass
consisting of lattice-ordered EAs includes the class of MV-algebras
\cite{CCC57}, the class of orthomodular lattices \cite{Kalm}, and
the class of boolean algebras.

In spite of their considerable generality, EAs have been further
generalized in two ways: (1) By dropping the assumption that
the orthosummation is commutative. (2) By dropping the assumption
that there is a unit. Noncommutative versions of EAs called
\emph{pseudo effect algebras} (PEAs), were introduced by A.
Dvure\v censkij and T. Vetterlein in \cite{DvVe1, DvVe2} and
further studied in \cite{DvVc, DVgen, Dvtat, DXY, DvKite, Dvnew, HS,
XLGRL}. The investigation of \emph{generalized effect algebras} (GEAs),
i.e., versions of EAs having no unit, was pioneered by Z. Rie\v canov\'a
in \cite{Zdenka99, RiMa, Ri08} and further studied in \cite{FPexoGEA,
PV}. Finally, by dropping both the assumption of commutativity and the
existence of a unit, one arrives at the notion of a \emph{generalized
pseudo effect algebra} (GPEA) \cite{DVgen, DVpo, FoPuUnit, XL, XLGRL}.

It is well known that every GEA $E$ can be embedded as a maximal proper
ideal in an EA $\widehat{E}$ called its \emph{unitization} in such a way
that (1) $E$ and $\widehat{E}\setminus E$ are order-anti-isomorphic under
the restrictions of the partial order on $\widehat{E}$; (2) for all
$a\in\widehat{E}$, either $a\in E$ or else its orthosupplement $a^{\perp}\in E$;
and (3) for $x,y\in\widehat{E}\setminus E$ the orthosum of $x$ and $y$ is not
defined \cite[Theorem 1.2.6]{DvPu}. This construction of a unitization was
extended to so-called weakly commutative GPEAs in \cite{XL}. Recently, in
\cite{FoPuUnit}, it was shown that a GPEA $P$ can be embedded as a maximal
proper PEA-ideal in a PEA $U$ if and only if $P$ admits a so-called unitizing
GPEA-automorphism. Indeed, the unitization of a weakly commutative GPEA
is a special case in which the unitizing GPEA-automorphism is the identity
mapping.

In this article, which is a continuation of \cite{FoPuUnit}, we
focus on properties of congruences and ideals in the setting of a
GPEA $P$ and its unitization $U$ with respect to a unitizing automorphism
$\gamma$ (a so-called $\gamma$-unitization). We find conditions under
which a congruence on $P$ can be extended to a congruence on its
$\gamma$-unitization $U$ such that the quotient of $U$ is the unitization
of the quotient of $P$ with a unitizing automorphism induced by $\gamma$.
In particular, we extend the results obtained in \cite{PV} and \cite{XL}
for unitizations of GEAs and weakly commutative GPEAs to the more general
unitizations of GPEAs with respect to a unitizing automorphism.

We briefly investigate several versions, RDP, RDP$_0$, RDP$_1$, RDP$_2$ of
the Riesz decomposition property in GPEAs and in their $\gamma$-unitizations
and show that if $P$ is a total GPEA, i.e., if the orthosum in $P$ is defined
for all pairs of its elements, then $P$ has one of these properties if and only
if its $\gamma$-unitization $U$ has the corresponding property.

We also study relations between the existence of a smallest nontrivial Riesz
ideal in a GPEA and the existence of a smallest nontrivial Riesz ideal in its
$\gamma$-unitization.

For the reader's convenience, we devote Sections \ref{sc:GPEAs} and
\ref{sc:Unitization} below to a brief review of some basic definitions
and facts needed in this article.

We  provide several illustrating examples. In \cite{DvKite} it was shown how to
construct a large class of PEAs by starting with the positive cone $G^+$ of
a po-group $G$, a nonempty indexing set $I$, and two bijections $\lambda,
\rho\colon I\to I$. In \cite{Dvnew}, this construction was extended by
replacing $G^+$ by a more general GPEA $E$. At the end of Section
\ref{sc:Unitization}, we review this construction and relate it to our work
in this article.

\section{Generalized pseudo effect algebras} \label{sc:GPEAs}

We begin this section by recalling the definition of a GPEA \cite
[Definition 2.6]{FoPuUnit} and we observe that PEAs, GEAs, and
EAs are special kinds of GPEAs. Axiomatic characterizations of PEAs,
GEAs, and EAs can be found in \cite[\S 2]{FoPuUnit}. Also we review
some of the basic properties of these partially ordered, partial
algebraic structures. We abbreviate `if and only if' by `iff' and
the symbol $:=$ means `equals by definition.'

\begin{definition}\label{de:GPEA} A \emph{generalized pseudo effect algebra}
(GPEA) \cite{DVgen, DVpo} is a partial algebraic structure $(P;\oplus,0)$,
where $\oplus$ is a partial binary operation on $P$ called the
\emph{orthosummation}, $0$ is a constant in $P$ called the \emph{zero element},
and the following conditions hold for all $a,b,c\in P$:
\begin{enumerate}
\item[ ]
 \begin{enumerate}
\item[(GPEA1)] $a\oplus b$ and $(a\oplus b)\oplus c$ exist iff $b\oplus c$
 and $a\oplus(b\oplus c)$ exist and in this case $(a\oplus b)\oplus c=a
\oplus(b\oplus c)$ (\emph{associativity}).
\item[(GPEA2)] If $a\oplus b$ exists, then there are elements $c,d\in P$
 such that $a\oplus b=c\oplus a=b\oplus d$ (\emph{conjugation}).
\item[(GPEA3)] If $a\oplus c=b\oplus c$, or $c\oplus a=c\oplus b$, then
 $a=b$ (\emph{cancellation}).
\item[(GPEA4)] $a\oplus 0=0\oplus a=a$ (\emph{neutral element}).
\item[(GPEA5)] If $a\oplus b=0$, then $a=0=b$ (\emph{positivity}).
\end{enumerate}
 \end{enumerate}
If no confusion threatens, we often denote the GPEA $(P;\oplus,0)$
simply by $P$. If we write an equation involving an orthosum of elements
of $P$ without explicitly assuming the existence thereof, we understand
that its existence is implicitly assumed.
\end{definition}

Let $P$ be a GPEA and let $a,b,c\in P$. By (GPEA3), the elements $c,d$ in
(GPEA2) are uniquely determined. A partial order $\leq$ is defined on
the GPEA $P$ (the \emph{induced partial order}) by stipulating that $a
\leq b$ iff there is (a necessarily unique) $c\in P$ with $a\oplus c=b$,
or equivalently, iff there is (a necessarily unique) $d\in P$ with $d
\oplus a=b$. For all $a\in P$, we have $0\leq a$. Moreover, the cancellation
laws (GPEA3) can be extended to $\leq$ as follows: \emph{If $a\oplus c\leq b
\oplus c$ or $c\oplus a\leq c\oplus b$, then $a\leq b$.}

Partial binary operations \emph{left subtraction} $\diagup$ and \emph{right
subtraction} $\diagdown$ are defined on $P$ as follows: For $a,b\in P$,
$a\diagup b$ and $b\diagdown a$ are defined iff $a\leq b$, in which case
(1) $a\diagup b:=c$, where $c$ is the unique element of $P$ such that
$a\oplus c=b$ and (2) $b\diagdown a:=d$, where $d$ is the unique element
of $P$ such that $d\oplus a=b$.

\begin{definition}
Let $P$ be a GPEA and let $I$ and $S$ be nonempty subsets of $P$. Then:
\begin{enumerate}
\item[(1)] $I$ is an \emph{order ideal} iff $a\in I$, $b\in P$, and $b\leq a$
 implies that $b\in I$.
\item[(2)] $I$ is an \emph{ideal} iff $I$ is an order ideal and whenever
 $a,b\in I$ and $a\oplus b$ is defined, it follows that $a\oplus b\in I$.
\item[(3)] An ideal $I$ in $P$ is said to be \emph{normal} iff whenever $a,b,c
 \in P$ and $a\oplus c=c\oplus b$, then $a\in I\Leftrightarrow b\in I$.
\item[(4)] $S$ is a \emph{sub-GPEA} of $P$ iff, whenever two of the elements
 $a,b,c\in P$ belong to $S$ and $a\oplus b=c$, then the third element also
 belongs to $S$.
\end{enumerate}
\end{definition}
Every ideal $I$ in $P$ is a sub-GPEA, and every sub-GPEA $S$ in $P$ is a GPEA
in its own right under the restriction to $S$ of the orthosummation on $P$.

\begin{lemma}\label{le:normalidealprops1}
Let $P$ be a GPEA, let $I$ be a normal ideal in $P$, let $a,b\in P$, and
suppose that $a\oplus b$ is defined. Then $b\in I$ iff $(a\oplus b)
\diagdown a\in I$; likewise, $a\in I$ iff $b\diagup(a\oplus b)\in I$.
\end{lemma}

\begin{proof}
Observe that $((a\oplus b)\diagdown a)\oplus a=a\oplus b$ and
$b\oplus(b\diagup(a\oplus b))=a\oplus b$.
\end{proof}

A GPEA $P$ is said to be \emph{total} iff $a\oplus b$ is defined for
all $a,b\in P$.

\begin{example} \label{ex:positivecone}
Let $G$ be an additively-written, not necessarily Abelian partially ordered
group {\rm(}po-group{\rm)}. Then the positive cone $G^+:=\{a\in G:0\leq a\}$
is a total GPEA with orthosummation $\oplus$ given by the restriction to $G^+$
of the group operation $+$ on $G$. In this case, the induced partial order
on the GPEA $G^+$ is the restriction to $G^+$ of the partial order on $G$.
\end{example}

If $P$ and $Q$ are GPEAs, then a mapping $\phi:P\to Q$ is a \emph{GPEA-morphism}
iff, for all $a,b\in P$, if $a\oplus b$ exists in $P$, then $\phi a\oplus\phi b$
exists in $Q$, and $\phi(a\oplus b)=\phi a\oplus\phi b$. A bijective GPEA-morphism
$\phi:P\to Q$ is a \emph{GPEA-isomorphism} of $P$ onto $Q$ iff $\phi^{-1}:Q\to
P$ is also a GPEA-morphism. A \emph{GPEA-automorphism} of $P$ is GPEA-isomorphism
$\phi:P\to P$.

It is not difficult to show that a bijective GPEA-morphism $\phi:P\to Q$ is a
GPEA-isomorphism iff, whenever $a,b\in P$ and $\phi a\oplus\phi b$ exists in
$Q$, then $a\oplus b$ exists in $P$.

A \emph{pseudo effect algebra} (PEA) is a GPEA with a largest element,
called the \emph{unit} and often denoted by $1$ (\cite[Definition
2.3]{FoPuUnit}). Let $P$ be a PEA with unit $1$. Clearly, the only element
$z\in P$ such that $z\oplus 1$ (or $1\oplus z$) exists is $z=0$. Also, for
$a\in P$, we have $a\leq 1$ whence we define $a\sp{\sim}:=a\diagup 1$ and
$a\sp{-}:=1\diagdown a$. Thus, $a\sp{\sim}$ and $a\sp{-}$, called the
\emph{right supplement} and the \emph{left supplement} of $a$, respectively,
are the unique elements in $P$ such that $a\oplus a\sp{\sim}=a\sp{-}\oplus a
=1$. If it happens that $a^{\sim}=a^-$, then the common element is called the
\emph{orthosupplement} of $a$ and is written as $a^{\perp}$. The PEA $P$ is
said to be \emph{symmetric} iff $a^{\perp}=a^{\sim}=a^-$ for all $a\in P$
\cite{DXY}.

\begin{example} \label{ex:intervalPEA}
Let $G$ be any {\rm(}additively written and not necessarily Abelian{\rm)}
po-group and choose an element $0\leq u\in G$. Let $G[0,u]:=\{a\in G:0
\leq a\leq u\}$, and define $a\oplus b$ for $a,b\in G[0,u]$ iff $a+b\leq u$,
in which case $a\oplus b:=a+b$. Then $(G[0,u];\oplus,0,u)$ is a PEA and
the induced partial order coincides with the partial order on $G$
restricted to $G[0,u]$. We note that, if $a,b\in G[0,u]$ and $a\leq b$,
then $a\sp{\sim}=-a+1$, $a\sp{-}=1-a$, $a\diagup b=-a+b$, and $b\diagdown a
=b-a$.
\end{example}

Some important basic properties of PEAs are collected in the following theorem
(see \cite[Theorem 2.4]{FoPuUnit}).

\begin{theorem}\label{basicpeas} Let $P$ be a PEA and let $a,b,c\in P$.
Then{\rm:}
\rm{(i)} $0^{\sim}=0^-=1$ and $1^{\sim}=1^-=0$. {\rm(iv)} $a\sp{\sim -}
=a\sp{-\sim}=a$. {\rm(ii)} $a\oplus b=c$ iff $b\sp{-}=c\sp{-}\oplus a$.
{\rm(vii)} $a\oplus b\sp{\sim}=c\sp{\sim}$ iff $c\oplus a=b$. {\rm(iii)}
$a^{\sim}\oplus b=c^{\sim}$ iff $b\sp{-}=c\oplus a\sp{\sim}$ iff $b^{--}\oplus
c=a$. {\rm(iv)} $a\leq b$ iff $b\sp{\sim}\leq a\sp{\sim}$ iff $b\sp{-}\leq a
\sp{-}$. {\rm(v)} Both $a\mapsto a\sp{\sim}$ and $a\mapsto a\sp{-}$ are
order-reversing bijections on $P$. {\rm(vi)} $a\oplus b$ exists iff $b\leq a
\sp{\sim}$ iff $a\leq b\sp{-}$.
\end{theorem}

In a straightforward way, one can derive formulas for left and right
subtraction as illustrated by the following lemma, the proof of which
we omit.

\begin{lemma} \label{lm:subtraction}
Let $P$ be a PEA and let $a,b\in P$. Then{\rm:}
{\rm(i)} $a\leq b\Rightarrow b\diagdown a=(a\oplus b\sp{\sim})\sp{-}$
and $a\diagup b=(b\sp{-}\oplus a)\sp{\sim}$. {\rm(ii)} $b\sp{--}\leq a
\Rightarrow b\sp{--}\diagup a=(a\sp{\sim}\oplus b)\sp{-}$. {\rm(iii)}
$b\leq a\sp{\sim}\Rightarrow a\leq b\sp{-}$, $a\sp{\sim}\diagdown b=
(b\sp{--}\oplus a)\sp{\sim}$, $b\sp{-}\diagdown a=(a\oplus b)\sp{-}$,
and $a\sp{\sim}\diagdown b=(b\sp{- -}\oplus a)\sp{\sim}$.
\end{lemma}

\begin{lemma}\label{le:normalidealprops2}
Let $P$ be a PEA, let $I$ be a normal ideal in $P$, and let $a,b\in P$.
Then{\rm:} {\rm(i)} $a\in I$ iff $a\sp{--}\in I$ iff $a\sp{\sim\sim}\in I$.
{\rm(ii)} $a\sp{-}\in I$ iff $a\sp{\sim}\in I$.
\end{lemma}

\begin{proof}
Part (i) follows from the facts that $a\sp{--}\oplus a\sp{-}=1=a\sp{-}
\oplus a$ and $a\oplus a\sp{\sim}=1=a\sp{\sim}\oplus a\sp{\sim\sim}$.
Part (ii) is a consequence of $a\oplus a\sp{\sim}=1=a\sp{-}\oplus a$.
\end{proof}

A \emph{state} on a PEA $P$ is a mapping $s:P\to[0,1]\subseteq{\mathbb R}$
such that (S1) $s(1)=1$ and (S2) $s(a\oplus b)=s(a)+s(b)$ whenever $a\oplus b$
exists in $P$. If $P$ is a PEA and $s$ is a state on $P$, then the \emph{kernel}
of $s$, i.e., $s\sp{-1}(0)=\{a\in P:s(a)=0\}$, is a normal ideal in $P$.

If $P$ and $Q$ are PEAs, then a mapping $\phi:P\to Q$ is a \emph{PEA-morphism}
iff it is a GPEA-morphism and $\phi 1=1$. The latter property is automatically
satisfied if $\phi$ is surjective. If $\phi:P\to Q$ is a bijective PEA-morphism
and $\phi^{-1}$ is also a PEA-morphism, then $\phi$ is a \emph{PEA-isomorphism}
of $P$ onto $Q$. A \emph{PEA-automorphism} of $P$ is a PEA-isomorphism
$\phi:P\to P$.

A GPEA (and in particular, a PEA) $P$ is said to be \emph{weakly commutative}
iff, for all $a,b\in P$, if $a\oplus b$ is defined, then $b\oplus a$ is
defined \cite{XLGRL}. It turns out that a PEA is symmetric iff it is weakly
commutative \cite{DXY, XL}.

Naturally, a GPEA $P$ is said to be \emph{commutative} iff, for all $a,b
\in P$, if $a\oplus b$ is defined, then $b\oplus a$ is defined, and then
$a\oplus b=b\oplus a$. A commutative GPEA is the same thing as a \emph
{generalized effect algebra} (GEA) \cite[Definition 2.2]{FoPuUnit} and
a commutative PEA is the same thing as an \emph{effect algebra} (EA)
\cite{FB}, \cite[Definition 2.1]{FoPuUnit}.

It can be shown that a GEA is total iff it can be realized, as per
Example \ref{ex:positivecone} as the positive cone in a directed Abelian
po-group.

A prototype for EAs is the system of all self-adjoint operators between
zero and identity on a Hilbert space $\mathfrak H$ (the system of so-called
\emph{effect operators} on $\mathfrak H$).

In many important examples, an effect algebra is an interval $G[0,u]=
\{a\in G:0\leq a\leq u\}$, as per Example \ref{ex:intervalPEA}, in the
positive cone of an Abelian po-group $G$. For instance, the set of all
effect operators on a Hilbert space $\mathfrak H$ is the interval
${\mathbb G}(\mathfrak H)[0,I]$ in the po-group ${\mathbb G}(\mathfrak H)$
(with the usual partial order) of all self-adjoint operators on $\mathfrak H$.
For more details about EAs and GEAs, see \cite{DvPu}.

\section{Unitization of a GPEA} \label{sc:Unitization}

In this section we review the notion of a binary unitization of a GPEA
and some of its properties \cite{FoPuUnit}.

\begin{definition}{\cite[Definition 3.1]{FoPuUnit}}\label{de:unitize} If $(P;\oplus,0)$
is a GPEA, then a PEA $(U;+,0,1)$ is called a binary \emph{unitization} of $P$ iff
the following conditions are satisfied:
\begin{enumerate}
\item[ ]
 \begin{enumerate}
\item[(U1)] $P\subseteq U$ and if $a,b\in P$, then $a\oplus b$
 exists in $P$ iff $a+b$ exists in $U$, and then $a\oplus b=a+b$.
\item[(U2)] $1\notin P$.
\item[(U3)] If $x,y\in U\setminus P$, then $x+y$ is undefined.
\end{enumerate}
 \end{enumerate}
 \end{definition}

In this paper, as in \cite{FoPuUnit}, \emph{we shall be considering
only binary unitizations}; hence, for simplicity, we usually omit the
adjective `binary' in what follows. See \cite[Theorem 3.3]{FoPuUnit}
for some of the basic properties of a unitization $U$ of a GPEA $P$;
in particular, for the fact that \emph{$P$ is a normal maximal proper
ideal of $U$}. By \cite[Theorem 5.3]{FoPuUnit}, \emph{a PEA $U$ is a
unitization of some GPEA $P$ iff $U$ admits a two-valued state $s$,
and in this case, $P$ is the kernel of $s$}.

\begin{definition} \label{df:unitizingauto}
If $P$ is a GPEA, then a GPEA-automorphism $\gamma:P\to P$ is said to
be \emph{unitizing} iff, for all $a,b\in P$, $\gamma a\oplus b$ is
defined iff $b\oplus a$ is defined.
\end{definition}

Obviously, if $P$ is a total GPEA, then every GPEA-automorphism of $P$
is unitizing. Also, $P$ is weakly commutative iff the identity mapping
on $P$ is a unitizing GPEA-automorphism. By \cite[Lemma 2.5, Lemma 2.7]
{FoPuUnit}, a PEA admits one and only one unitizing GPEA-automorphism, namely
$a\mapsto a^{--}$.

The next theorem describes the construction of a unitization of a GPEA $P$
having a unitizing automorphism $\gamma$.

\begin{theorem}{\rm \cite[Theorem 4.2]{FoPuUnit}} \label{th:gammaunitization} Let
$(P;\oplus,0)$ be a GPEA, let $\gamma:P\to P$ be a unitizing GPEA-automorphism,
let $P^{\eta}$ be a set disjoint from $P$ and with the same cardinality as $P$,
and let $\eta:P\to P^{\eta}$ be a bijection. Define $U:=P\cup P^{\eta}$ and let
$+$ be the partial binary operation on $U$ defined as follows:
\begin{enumerate}
\item[\rm(1)] If $a,b\in P$, then $a+b$ is defined iff $a\oplus b$ is defined, in
 which case $a+b:=a\oplus b$.
\item[\rm(2)] If $a\in P$ and $x\in U\setminus P$ with $b:=\eta^{-1}x$, then
 $a+x$ is defined iff $a\leq b$, in which case $a+x:=\eta c\in P^{\eta}=
 U\setminus P$, where $c$ is the unique element of $P$ such that $c\oplus a=b$.
 Thus, for $a,b\in P$, $a+\eta b$ is defined iff $a\leq b$, in which case $a+
 \eta b=\eta(b\diagdown a)\in U\setminus P$.
\item[\rm(3)] If $b\in P$ and $y\in U\setminus P$ with $a:=\eta^{-1}y$, then $y+b$
 is defined iff $\gamma b\leq a$, in which case $y+b:=\eta c\in P^{\eta}=U
 \setminus P$ where $c$ is the unique element of $P$ such that $\gamma b\oplus c=a$.
 Thus, for $a,b\in P$, $\eta a+b$ is defined iff $\gamma b\leq a$, in which case
 $\eta a+b=\eta(\gamma b\diagup a)\in U\setminus P$.
\item[\rm(4)] If $x,y\in P^{\eta}=U\setminus P$, then $x\oplus y$ is undefined.
\end{enumerate}
Then, with $1:=\eta 0$, $(U;+,0,1)$ is a PEA and it is a unitization of the
GPEA $(P;\oplus,0)$. Moreover, for $a\in P$, we have $\eta a=a\sp{\sim}\in U
\setminus P$ and $\gamma a=a\sp{--}\in P$.
\end{theorem}

We shall refer to the unitization $U$ of $P$ as constructed in Theorem
\ref{th:gammaunitization} by using the unitizing automorphism $\gamma$ as the
\emph{$\gamma$-unitization} of $P$.

Suppose that $V$ is a unitization of the GPEA $P$ and that $\gamma
\colon P\to P$ is the restriction to $P$ of the mapping $x\mapsto
x\sp{--}$ on $V$. Then by \cite[Theorem 3.3 (viii)]{FoPuUnit}, $\gamma$
is a unitizing GPEA-automorphism of $P$ called \emph{the unitizing
GPEA-automorphism of $P$ corresponding to $V$}. Therefore, \emph{a
GPEA $P$ admits a unitization iff it admits a unitizing GPEA-automorphism.}

The proof of the following theorem, using the results of \cite[Theorem 3.3]
{FoPuUnit}, is fairly straightforward, and we omit it here.

\begin{theorem}
Let $V$ be a unitization of the GPEA $P$, let $\gamma\colon P\to P$
be the corresponding unitizing GPEA-automorphism, and let $U$ be the
$\gamma$-unitization of $P$. Then there exists a unique PEA-isomorphism
of $U$ onto $V$ that reduces to the identity on $P$, namely the
mapping $\phi\colon U\to V$ defined by $\phi a=a$ for $a\in P$ and
$\phi x:=(x\sp{-\sb{U}})\sp{\sim\sb{V}}$ for $x\in U\setminus P$.
\end{theorem}

\begin{example}\label{ex:2.13} {\rm \cite[Examples 2.13 and 3.2]{DXY}} Let
$G$ be a directed po-group with center $C(G)$, let $c\in G$, let $\integers
\overset{\rightarrow}\times G$ be the lexicographic product of the ordered
group of integers $\integers$ with $G$, and let $0<n\in\integers$. Define
the interval PEA $U:=(\integers\overset{\rightarrow}\times G)[(0,0),(n,c)]$.
Then $U$ is symmetric {\rm(}weakly commutative{\rm)} iff $c\in C(G)$.

Now we consider the case $n=1$ and $c\notin C(G)$. Then the PEA $U=(\integers
\overset{\rightarrow}\times G)[(0,0),(1,c)]$ is not symmetric. Define a
two-valued state $s$ on $U$ by
\begin{enumerate}
\item[{\rm(1)}] $s(0,g)=0$ if $0\leq g\in G$,\  and\  {\rm(2)} $s(1,g)=1$ if $c\geq g\in G$.
\end{enumerate}
Then by {\rm\cite[Theorem 5.3]{FoPuUnit}}, the set $P:=s^{-1}(0)=\{(0,g):0\leq g\}$
forms a normal maximal proper ideal in $U$, $P$ is a GPEA isomorphic to the
GPEA $G^+$ {\rm(}Example \ref{ex:positivecone}{\rm)}, and $U$ is a unitization
of $P$. Note that $P$ is a total GPEA, hence $P$ is weakly commutative. To find
the unitizing GPEA-automorphism $\gamma$, we observe that
$$
(0,g)^-+(0,g)=(1,c)\,\Rightarrow\, (0,g)^-=(1,c-g), \text{\ and}
$$
$$
(0,g)^{--}+(1,c-g)=(1,c)\,\Rightarrow\, (0,g)^{--}=(0,c-(c-g)).
$$
Thus for the element $(0,g)\in P$, $\gamma(0,g)=(0,c-(c-g))$.
\end{example}

We conclude this section with a presentation of an important class of
PEAs, called \emph{kite algebras}, constructed by A. Dvure\v{c}enskij in
\cite{Dvnew}. As we shall see, every kite algebra is a unitization of a
GPEA. In our presentation, we shall find it convenient to make some small
changes in the notation of \cite{Dvnew}. First, for consistency with our
notation above, we use $P$, rather than $E$, for the base GPEA from which
a kite algebra is constructed. Second, for the bijections $\lambda$ and
$\rho$ in \cite{Dvnew}, we change notation to reduce the number of
occurrences of $\lambda\sp{-1}$ and $\rho\sp{-1}$. Third, for ease of
comparison with Theorem \ref{th:gammaunitization} we replace
${\overline a}$ by $\eta a$. Fourth, we use $K$ rather than $K
\sp{\lambda,\rho}\sb{I}$ for a kite algebra.

Thus, for the remainder of this section \emph{we assume that $(P;
\oplus,0)$ is a GPEA, $I$ is a nonempty indexing set, and $\lambda,
\rho\colon I\to I$ are bijections. We organize $P\sp{I}$ into a GPEA
$(P\sp{I};\oplus\sb{P\sp{I}},0\sp{I})$ where $\oplus\sb{P\sp{I}}$ is
the obvious coordinatewise orthosummation and $0\sp{I}$ is the family
in $P\sp{I}$ all the elements of which are $0$. Let $P\sp{\eta}$ be
a set disjoint from $P$ and with the same cardinality as $P$ and let
$\eta\colon P\to P\sp{\eta}$ be a bijection.}

Evidently, $P$ is total (respectively, weakly commutative) iff
$P\sp{I}$ is total (respectively, weakly commutative).

The following conditions on the bijections $\lambda,\rho\colon I\to I$
were introduced in \cite{Dvnew} (but with $\lambda$ and $\rho$
replaced by $\lambda\sp{-1}$ and $\rho\sp{-1}$): For all $(a\sb{i})
\sb{i\in I},(b\sb{i})\sb{i\in I}\in P\sp{I}$ and for all $i\in I$,
\begin{enumerate}
\item[ ]
 \begin{enumerate}
\item[(KCI)] $a\sb{\rho i}\oplus b\sb{i}$ exists in $P$ iff
 $b\sb{i}\oplus  a\sb{\lambda i}$ exists in $P$, and
\item[(KCII)] $a\sb{\lambda i}\oplus b\sb{i}$ exists in $P$ iff
  $b\sb{i}\oplus a\sb{\rho i}$ exists in $P$.
\end{enumerate}
 \end{enumerate}
Obviously, if $\lambda=\rho$, then conditions (KCI) and (KCII) are
identical, and in this case it can be shown that they both hold iff
$P$ is a total GPEA. On the other hand, if $P$ is a total GPEA,
then (KCI) and(KCII) are automatically satisfied. In \cite{Dvnew}, it
is assumed that the bijections $\lambda,\rho\colon I\to I$ satisfy
both conditions (KCI) and (KCII). Later in Lemma \ref{lm:gammaunitizing}
and Theorem \ref{th:kiteisomorphism}, we shall assume (KCI), but not
necessarily (KCII).

\begin{definition}
Noting that $P\sp{I}$ and $(P\sp{\eta})\sp{I}$ are disjoint, we
define
\[
K:=P\sp{I}\cup(P\sp{\eta})\sp{I}
\]
and we define $0\sb{K}\in P\sp{I}$ and $1\sb{K}\in(P\sp{\eta})\sp{I}$
by
\[
0\sb{K}:=0\sp{I}\text{\ and\ }1\sb{K}=(\eta c\sb{i})\sb{i\in I},
 \text{\ where\ }c\sb{i}:=0\text{\ for all\ }i\in I.
\]
We organize $K$ into a partial algebra $(K;+,0\sb{K},
1\sb{K})$, where the partial binary operation $+$ on $K$ is
defined as follows: If $(a\sb{i})\sb{i\in I}, (b\sb{i})\sb{i\in I}
\in P\sp{I}$, then
\begin{enumerate}
\item[ ]
 \begin{enumerate}
\item[{\rm(K1)}] $(a\sb{i})\sb{i\in I}+(b\sb{i})\sb{i\in I}:=(a\sb{i}
 \oplus b\sb{i})\sb{i\in I}$ iff $a\sb{i}\oplus b\sb{i}$ exists in
 $P$ for all $i\in I$.
\item[{\rm(K2)}] $(a\sb{i})\sb{i\in I}+(\eta b\sb{i})\sb{i\in I}:=
 (\eta(b\sb{i}\diagdown a\sb{\lambda i}))\sb{i\in I}$ iff $a\sb
 {\lambda i}\leq b\sb{i}$ for all $i\in I$.
\item[{\rm(K3)}] $(\eta a\sb{i})\sb{i\in I}+(b\sb{i})\sb{i\in I}:=
 (\eta(b\sb{\rho i}\diagup a\sb{i}))\sb{i\in I}$ iff $b\sb{\rho i}
 \leq a\sb{i}$ for all $i\in I$.
\item[{\rm(K4)}] $(\eta a\sb{i})\sb{i\in I}+(\eta b\sb{i})\sb{i\in I}$
 is undefined.
\end{enumerate}
 \end{enumerate}
\end{definition}
\noindent If $K$ is a PEA, it is called the \emph{kite algebra} determined
by $P$, $I$, $\lambda$, and $\rho$.

\smallskip

By (K1), the restriction of $+$ to the GPEA $P\sp{I}$ coincides with
$\oplus\sb{P\sp{I}}$, whence, for all $(a\sb{i})\sb{i\in I}$, we have
 $(a\sb{i})\sb{i\in I}+0\sb{K}=0\sb{K}+(a\sb{i})\sb{i\in I}$. Also,
by (K2) and (K3), $0\sb{K}+(\eta a\sb{i})\sb{i\in I}=(\eta a\sb{i})
\sb{i\in I}+0\sb{K}=(\eta a\sb{i})\sb{i\in I}$. Furthermore, by
(K2) and (K3), $(a\sb{i})\sb{i\in I}+1\sb{K}$ is defined iff
$1\sb{K}+(a\sb{i})\sb{i\in I}$ is defined iff $(a\sb{i})\sb{i\in I}=
0\sb{K}$, whereas by (K4), $(\eta a\sb{i})\sb{i\in I}+1\sb{K}$
and $1\sb{K}+(\eta a\sb{i})\sb{i\in I}$ are undefined. If $K$ is a
kite algebra, it is obviously a (binary) unitization of $P\sp{I}$.

\begin{definition} \label{df:gamma}
Define $\gamma\colon P\sp{I}\to P\sp{I}$ by
$$\gamma((a\sb{i})\sb{i\in I})=(a\sb{\rho\lambda^{-1}i})\sb{i\in I}\text
 {\ for all\ }(a\sb{i})\sb{i\in I}\in P\sp{I}.$$
\end{definition}

\begin{lemma} \label{lm:gammaunitizing}
{\rm(KCI)} holds iff $\gamma$ is a unitizing GPEA-automorphism on
$P\sp{I}$.
\end{lemma}

\begin{proof}
Clearly, $\gamma$ is a GPEA-automorphism on $P\sp{I}$. Assume that
(KCI) holds. Replacing $i$ by $\lambda\sp{-1}i$ in (KCI), we find that,
for all $(a\sb{i})\sb{i\in I},(b\sb{i})\sb{i\in I}\in P\sp{I}$
and for all $i\in I$,
\setcounter{equation}{0}
\begin{equation} \label{eq:gamma1}
a\sb{\rho\lambda\sp{-1}i}\oplus b\sb{\lambda\sp{-1}i} \text{\ exists in\ }P
\text{\ iff\ }b\sb{\lambda\sp{-1}i}\oplus a\sb{i}\text{\ exists in\ }P.
\end{equation}
Now, given any $(c\sb{i})\sb{i\in I}\in P\sp{I}$, we let $b\sb{i}
:=c\sb{\lambda i}$ in (\ref{eq:gamma1}), so that $b\sb{\lambda\sp{-1}i}
=c\sb{i}$, and we have, for all $(a\sb{i})\sb{i\in I},(c\sb{i})\sb{i\in I}
\in P\sp{I}$ and for all $i\in I$,
\begin{equation} \label{eq:gamma2}
a\sb{\rho\lambda\sp{-1}i}\oplus c\sb{i}\text{\ exists in\ }P\text{\ iff\ }
c\sb{i}\oplus a\sb{i}\text{\ exists in\ }P.
\end{equation}
Therefore, $\gamma$ is a unitizing GPEA-automorphism on $P\sp{I}$.
Conversely, if (\ref{eq:gamma2}) holds, then by first replacing $i$ by
$\lambda i$, then putting $b\sb{i}:=c\sb{\lambda i}$, we arrive back
at (KCI).
\end{proof}

Suppose that (KCI) holds, so that, by Lemma \ref{lm:gammaunitizing},
$\gamma$ is a unitizing GPEA-auto\-morphism on $P\sp{I}$. Therefore, we
can construct the $\gamma$-unitization $U$ of $P\sp{I}$ as per Theorem
\ref{th:gammaunitization}. To do this, we begin by putting $(P\sp{I})
\sp{\eta}:=(P\sp{\eta})\sp{I}$ and we choose for the bijection from
$P\sp{I}$ to $(P\sp{I})\sp{\eta}$ the mapping---also denoted by
$\eta$---defined by $\eta(a\sb{i})\sb{i\in I}:=(\eta a\sb{i})\sb{i\in I}$
for all $(a\sb{i})\sb{i\in I}\in P\sp{I}$. (This dual use of $\eta$ should
not cause any confusion.) Then, \emph{as a set}, the $\gamma$-unification
$U=P\sp{I}\cup(P\sp{I})\sp{\eta}$ of $P\sp{I}$ is the same as $K$. However,
as per Theorem \ref{th:gammaunitization} and the definition of $\gamma$,
$U$ is organized into a PEA $(U;+\sb{U},0\sb{U},1\sb{U})$ as follows:
$0\sb{U}:=0\sb{K}$, $1\sb{U}:=1\sb{K}$, and for $(a\sb{i})\sb{i\in I},
(b\sb{i})\sb{i\in I}$ in $P^I$:
\begin{enumerate}
\item[(U1)] $(a_i)_{i\in I}+\sb{U}(b_i)_{i\in I}:=(a_i\oplus b_i)_{i\in I}$ iff
 $a_i\oplus b_i$ is defined in $P$ for all $i\in I$.
\item[(U2)] $(a_i)_{i\in I}+\sb{U}(\eta b_i)\sb{i\in I}:=(\eta(b_i\diagdown a_i))
 _{i\in I}$ iff $a\sb{i}\leq b\sb{i}$ for all $i\in I$.
\item[(U3)] $(\eta a\sb{i})\sb{i\in I}+\sb{U}(b\sb{i})\sb{i\in I}:=
 (\eta(b\sb{\rho\lambda\sp{-1}i}\diagup a\sb{i}))\sb{i\in I}$ iff $b\sb{\rho
 \lambda\sp{-1}i}\leq a\sb{i}$ for all $i\in I$.
\item[(U4)] $(\eta a\sb{i})\sb{i\in I}+\sb{U}(\eta b\sb{i})\sb{i\in I}$ is
 undefined.
\end{enumerate}
According to (U1), the restriction of $+\sb{U}$ to $P\sp{I}$ coincides with
$\oplus\sb{P\sp{I}}$, whence it also coincides with the restriction of $+$
to $P\sp{I}$.

\begin{theorem} \label{th:kiteisomorphism}
Suppose that {\rm(KCI)} holds. Then the mapping $\phi\colon U\to K$ defined
for all $(a\sb{i})\sb{i\in I}\in P\sp{I}$ by
\[
\phi(a\sb{i})\sb{i\in I}:=(a\sb{i})\sb{i\in I}\text{\ and\ }\phi((\eta a
 \sb{i})\sb{i\in I})=(\eta a\sb{\lambda i})\sb{i\in I}
\]
is an isomorphism of the PEA $(U;+\sb{U},0\sb{U},1\sb{U})$ onto the
partial algebra $(K;\linebreak+,0\sb{K},1\sb{K})$ such that the
restriction of $\phi$ to $P\sp{I}$ is the identity mapping. Therefore,
$(K;+,0\sb{K},1\sb{K})$ is a PEA, $K$ is a kite algebra, and $\phi
\colon U\to K$ is a PEA-isomorphism. Furthermore, if $\psi\colon U
\to K$ is any PEA-morphism such that the restriction of $\psi$ to
$P\sp{I}$ is the identity mapping, then $\psi=\phi$.
\end{theorem}

\begin{proof}
Obviously, the restriction of $\phi$ to $P\sp{I}$ is the identity
mapping, $\phi 0\sb{U}=0\sb{K}$, $\phi 1\sb{U}=1\sb{K}$, and
$\phi\colon U\to K$ is a bijection. In fact, for all $(a\sb{i})
\sb{i\in I}\in P\sp{I}$, $\phi\sp{-1}(a\sb{i})\sb{i\in I}:=
(a\sb{i})\sb{i\in I}$ and $\phi\sp{-1}((\eta a\sb{i})\sb{i\in I})
=(\eta a\sb{\lambda\sp{-1}i})\sb{i\in I}.$
To complete the proof that $\phi\colon U\to K$ is a PEA-isomorphism,
we have to prove that
\begin{enumerate}
\item If $p,q\in U$ and $p+\sb{U}q$ is defined in $U$, then
 $\phi p+\phi q$ is defined in $K$ and $\phi(p+\sb{U}q)=
 \phi p+\phi q$ and
\item If $s,t\in K$ and $s+t$ is defined in $K$, then
 $\phi\sp{-1}s+\sb{U}\phi\sp{-1}t$ is defined in $U$ and $\phi\sp{-1}
(s+t)=\phi\sp{-1}s+\sb{U}\phi\sp{-1}t$.
\end{enumerate}

To prove condition (i), assume that $p,q\in U$. There are only two
nontrivial cases to consider.

\noindent \emph{Case 1}: $p=(a\sb{i})\sb{i\in I}\in P\sp{I}$ and $q=
(\eta b\sb{i})\sb{i\in I}\in(P\sp{\eta})\sp{I}$.

\noindent \emph{Case 2}: $p=(\eta a\sb{i})\sb{i\in I}\in(P\sp{\eta})\sp{I}$
and $q=(b\sb{i})\sb{i\in I}\in P\sp{I}$.

In Case 1, suppose that $p+\sb{U}q$ is defined in $U$. Then by (U2),
$a\sb{i}\leq b\sb{i}$, whence $a\sb{\lambda i}\leq b\sb{\lambda i}$ for all
$i\in I$, and
$$p+\sb{U}q=(\eta(b\sb{i}\diagdown a\sb{i}))\sb{i\in I}\text{\ for all\ }
i\in I,\text{\ so}$$
\[
\phi p=(a\sb{i})\sb{i\in I},\  \phi q=(\eta b\sb{\lambda i})\sb{i\in I},\
\phi(p+\sb{U}q)=(\eta(b\sb{\lambda i}\diagdown a\sb{\lambda i}))\sb{i\in I},
\]
and by (K2),
$$\phi p+\phi q=(a\sb{i})\sb{i\in I}+(\eta b\sb{\lambda i})\sb{i\in I}:=
(\eta(b\sb{\lambda i}\diagdown a\sb{\lambda i})\sb{i\in I}=\phi(p+\sb{U}q).$$

In Case 2, suppose that $p+\sb{U}q$ is defined in $U$. Then by (U3),
$b\sb{\rho\lambda\sp{-1}i}\leq a\sb{i}$, whence $a\sb{\rho i}\leq b\sb
{\lambda i}$ for all $i\in I$, and
$$p+\sb{U}q=(\eta(b\sb{\rho\lambda\sp{-1}i}\diagup a\sb{i})\sb{i\in I}
\text{\ for all\ }i\in I, \text{\ so}$$
$$\phi p=(\eta a\sb{\lambda i})\sb{i\in I},\ \phi q=(b\sb{i})\sb{i\in I},
\ \phi(p+\sb{U}q)=(\eta(b\sb{\rho i}\diagup a\sb{\lambda i}))\sb{i\in I},$$
and by (K3),
$$\phi p+\phi q=(\eta a\sb{\lambda i})\sb{i\in I}+(b\sb{i})\sb{i\in I}=
(\eta(b\sb{\rho i}\diagup a\sb{\lambda i})\sb{i\in I}=\phi(p+\sb{U}q),$$
completing the proof of (i).

The proof of (ii) is a straightforward computation similar to the proof of
(i), and is therefore omitted. Thus, we may conclude that $\phi$ is an
isomorphism of $U$ onto $K$ that reduces to the identity on $P\sp{I}$;
hence $K$ is a PEA.

Finally, suppose that $\psi\colon U\to K$ is any PEA-morphism such that the
restriction of $\psi$ to $P\sp{I}$ is the identity mapping. According to
(U4) with $b\sb{i}=a\sb{i}$ for all $i\in I$, for all $(a\sb{i})\sb{i\in I}
\in P\sp{I}$,
$$(a\sb{i})\sb{i\in I}+\sb{U}(\eta a\sb{i})\sb{i\in I}=1\sb{U},\text
{\ whence\ }(a\sb{i})\sb{i\in I}+\psi(\eta a\sb{i})\sb{i\in I}=1\sb{K}.$$
But by (K2), we also have
$(a\sb{i})\sb{i\in I}+(\eta a\sb{\lambda i})\sb{i\in I}=1\sb{K},$
and it follows from cancellation that
$$\psi(\eta a\sb{i})\sb{i\in I}=(\eta a\sb{\lambda i})\sb{i\in I}=
\phi(\eta a\sb{i})\sb{i\in I}.$$
Therefore, $\psi=\phi$.
\end{proof}

The proof of the following corollary is now straightforward.

\begin{corollary}\label{co:KCI}
Suppose that {\rm(KCI)} holds. Then, for all $(a\sb{i})\sb{i\in I}$, the
left and right negations on the kite algebra $K$ are as follows{\rm:}
\begin{enumerate}
\item[ ]
 \begin{enumerate}
\item[{\rm(LN)}] $((a\sb{i})\sb{i\in I})\sp{-}=(\eta a\sb{\rho i})
 \sb{i\in I}$ and $((\eta a\sb{i})\sb{i\in I})\sp{-}=(a\sb{\lambda
 \sp{-1}i})\sb{i\in I}$,
\item[{\rm(RN)}] $((a\sb{i})\sb{i\in I})\sp{\sim}=(\eta a\sb{\lambda i})
 \sb{i\in I}$ and $((\eta a\sb{i})\sb{i\in I})\sp{\sim}=(a\sb{\rho
 \sp{-1}i})\sb{i\in I}$,
 \end{enumerate}
\end{enumerate}
and from {\rm(LN)}, we have
\begin{enumerate}
\item[{ }]$((a\sb{i})\sb{i\in I})\sp{--}=(a\sb{\rho\lambda\sp{-1}i})\sb{i\in I}
 =\gamma(a\sb{i})\sb{i\in I}$.
\end{enumerate}
Therefore, $(K;+,0\sb{K},1\sb{K})$ is a {\rm(}binary{\rm)} unitization of
$P\sp{I}$ with $\gamma$ as its unitizing PEA-automorphism.
\end{corollary}

If $G$ is an additively written po-group and if we put $P:=G\sp{+}$
{\rm(}Example \ref{ex:positivecone}{\rm)}, then $P$ is total and the PEA $K$
is what A. Dvure\v{c}enskij originally called a kite PEA in {\rm\cite{DvKite}}.

\section{Congruences and ideals in GPEAs and in their unitizations}

In what follows, we will need some facts about congruences and ideals in GPEAs.

\begin{definition}\label{de:congr}{\rm \cite[Definition 2.1]{XL}} A binary
relation $\sim$ on a GPEA $P$ is called a \emph{weak congruence} iff it
satisfies the following conditions:
\begin{enumerate}
\item[(C1)] $\sim$ is an equivalence relation.
\item[(C2)] If $a\oplus b$ and $a_1\oplus b_1$ both exist, $a\sim a_1$,
 and $b\sim b_1$, then $a\oplus b\sim a_1\oplus b_1$.
\end{enumerate}
A weak congruence is a \emph{congruence} iff it satisfies the following condition:
\begin{enumerate}
\item[(C3)] If $a\oplus b$ exists, then for any $a_1\sim a$ there is a $b_1\sim b$
 such that $a_1\oplus b_1$ exists, and for any $b_2\sim b$ there is an $a_2\sim a$
 such that $a_2\oplus b_2$ exists.
\end{enumerate}
A congruence $\sim$ is called a \emph{c-congruence} iff it satisfies the following
condition:
\begin{enumerate}
\item[(C4)] If $a\sim b$ and either $a\oplus a_1\sim b\oplus b_1$ or
 $a_1\oplus a\sim b_1\oplus b$, then $a_1\sim b_1$.
\end{enumerate}
A congruence $\sim$  is called a \emph{p-congruence} iff it satisfies the following
 condition:
\begin{enumerate}
\item[(C5)] $a\oplus b\sim 0$ implies $a\sim b\sim 0$.
\end{enumerate}
\end{definition}

Let $\sim$ be a weak congruence. Denote by $[a]$ the congruence class containing
$a\in P$, and let $P/{\sim}$ denote the set of all congruence classes (the
\emph{quotient} of $P$ with respect to $\sim$). We shall say that $[a]\oplus [b]$
exists iff there are $a_1,b_1\in P$ with $a_1\sim a,\ b_1\sim b$, such that
$a_1\oplus b_1$ exists, in which case, $[a]\oplus [b]:=[a_1\oplus b_1]$.

\begin{theorem}\label{th:congr} {\rm \cite[Corollary 2.1]{XL}} For a congruence
$\sim$ on a GPEA $P$, the quotient $(P/{\sim}; \oplus,[0])$ is a GPEA iff $\sim$
is  both a c-congruence and a p-congruence.
\end{theorem}

If $P$ is a PEA, then every congruence on $P$ is a c-congruence and p-congruence,
as can be seen from the following theorem.

\begin{theorem}\label{th:congrpea} {\rm \cite{HS}} A weak congruence $\sim$ on a PEA
$P$ is a congruence iff the following conditions hold:
\begin{enumerate}
\item[{\rm(C4$\sp{\prime}$)}] If $a\sim b$, then $a^{\sim}\sim b^{\sim}$ and $a^-\sim b^-$.
\item[{\rm(C5$\sp{\prime}$)}] $a\sim b\oplus c$ implies that there are $a_1,a_2\in P$,
 such that $a_1\sim b,\ a_2\sim c$ and $a=a_1\oplus a_2$.
\end{enumerate}
Moreover, {\rm(C4$\sp{\prime}$)} is equivalent to {\rm(C4)}, and {\rm(C5$\sp{\prime}$)}
implies {\rm(C5)}.
\end{theorem}

Of course, if $P$ is a PEA and $\sim$ is a congruence on $P$, then $P/\sim$ is
also a PEA.

\begin{definition}\label{de: riesz} An ideal $I$ in a GPEA $P$ is called an
\emph{R1-ideal} iff the following condition holds:
\begin{enumerate}
\item[(R1)] If $i\in I$, $a,b\in P$, $a\oplus b$ exists, and $i\leq a\oplus b$, then
 there  are $j,k\in I$ such that $j\leq a,\  k\leq b$ and $i\leq j\oplus k$.
\end{enumerate}
An R1-ideal $I$ is called a \emph{Riesz ideal} iff the following condition holds:
\begin{enumerate}
\item[(R2)] If $i\in I$, $a,b\in P$, and $i\leq a$, then (i) if $(a\diagdown i)
 \oplus b$ exists, then there is $j\in I$ such that $j\leq b$ and $a\oplus(j\diagup b)$
 exists, and (ii) if $b\oplus(i\diagup a)$ exists, then there is $k\in I$ such that $k
 \leq b$ and $(b\diagdown k)\oplus a$ exists.
\end{enumerate}
\end{definition}

\begin{definition} \label{df:simsbI}
For an ideal $I$ in a GPEA $P$, we define $a\sim\sb{I}b$ ($a,b\in P$) iff there
exist $i,j\in I$, $i\leq a,\  j\leq b$ such that $a\diagdown i=b\diagdown j$.
\end{definition}
Notice that if $I$ is a normal ideal in the GPEA $P$, then the condition in
Definition \ref{df:simsbI} that $a\diagdown i=b\diagdown j$ with $i,j\in I$ is
equivalent to the condition that $i\diagup a=j\diagup b$ (see \cite[Remark 2.4]{XL}).
Also, if $i\in I$ and $i\leq a\in P$, then $(a\diagdown i)\diagdown 0=a\diagdown i$,
so $a\diagdown i\sim\sb{I} a$.

\begin{theorem}\label{th:R1congr}{\rm\cite[Theorem 2.3]{XL}} If $I$ is a normal
R1-ideal in a GPEA $P$, then $\sim_I$  satisfies {\rm(C1), (C2)} and
{\rm(C5$\sp{\prime}$)}. Moreover, $a\sim_I0$ iff $a\in I$.
\end{theorem}

\begin{theorem}\label{th:rieszcongr}{\rm \cite[Corollary 2.2]{XL}} If $I$ is a
normal Riesz ideal in a GPEA $P$, then $\sim_I$ is a c-congruence and p-congruence,
hence $P/{\sim_I}$ is a GPEA.
\end{theorem}

Recall that a poset $V$ is \emph{upward directed} iff for any $a,b\in V$ there is
$c\in V$ with $a,b\leq c$. Similarly $V$ is \emph{downward directed} iff for any
$a,b\in V$ there is $d\in V$ with $d\leq a,b$.

\begin{theorem}\label{th:upwardR1}{\rm \cite[Proposition 2.2]{XL}} In an upward
directed GPEA $P$, an ideal $I$ is a Riesz ideal iff $I$ is an R1-ideal.
\end{theorem}

\begin{definition} \label{df:Rieszcongruence}
We say that a congruence $\sim$ on a GPEA $P$ is a \emph{Riesz congruence} iff it
satisfies {\rm(C4), (C5$\sp{\prime}$)} and the following condition:
\begin{enumerate}
\item[(CR)] If $a\sim b$, then there are $c,d\in P$, such that $c\leq a\leq d$,
 $c\leq b\leq d$, $a\diagdown c\sim b\diagdown c\sim 0$, and $d\diagdown a\sim d
 \diagdown b\sim 0$.
\end{enumerate}
\end{definition}

\begin{theorem}\label{th:upward}{\rm \cite[Lemma 2.4]{XL}} If $I$ is a normal Riesz
ideal in an upward directed GPEA $P$, then $\sim\sb{I}$ is a Riesz congruence.
\end{theorem}

\begin{theorem}\label{th:newtheor} Let $P$ be a GPEA. Then{\rm:}

{\rm(i)} A congruence $\sim$ on $P$ satisfying {\rm(C4), (C5$\sp{\prime}$)}
is a Riesz congruence iff  every equivalence class is both downward and
upward directed.

{\rm(ii)} If $\sim$ is a Riesz congruence on $P$, then $I:=\{i\in P:i\sim 0\}$
is a normal Riesz ideal in $P$. Moreover, $\sim =\sim_I$.
\end{theorem}

\begin{proof} (i) If $\sim$ is a Riesz congruence, then every equivalence class
is upward and downward directed by \cite[Proposition 2.8 (1)]{XL}. Conversely, assume
that $\sim$ is a congruence such that every equivalence class is downward and upward
directed and suppose that $a\sim b$. Then there exist $c,d$ such that $c\leq a\leq d$,
$c\leq b\leq d$, and $c\sim a\sim b\sim d$. By (C4) we then obtain $a\diagdown c\sim b
\diagdown c\sim 0$, and $d\diagdown a \sim d\diagdown b\sim 0$. Hence $\sim$ satisfies
(CR).

(ii) If $x\in I$ and $y\leq x$, then $(x\diagdown y)\oplus y\sim 0$. Since $\sim$ is a
p-congruence, $y\sim 0$ and $y\in I$. If $i\sim 0, j\sim 0$ and $i\oplus j$ exists,
then $i\oplus j\sim 0$. Thus $I$ is an ideal.  Assume $i\in I$, $a,j\in P$ and
$a\oplus i=j\oplus a$. Then $0\oplus a\sim a\oplus i=j\oplus a$, and by (C4) we have
$0\sim j$, whence $j\in I$. Thus $I$ is a normal ideal.

The proof of the R1 property is the same as in the proof of \cite[Lemma 2.3]{XL}.

To prove R2, assume that $a,b\in P,\ i\in I$, and $(a\diagdown i)\oplus b$ exists.
Thus $a\diagdown i\sim a$ and by (C3), there is $b_1\sim b$ such that $a\oplus b_1$
exists. Then there is $b_2\leq b,b_1$ with $b_2\sim b$. From this we get $b_2=
j\diagup b$ and $a\oplus b_2$ exists. Finally, from $b_2=j\diagup b\sim b$ we get
by (C4) that $j\sim 0$. To prove the remaining statement, assume first that $a\diagdown i=b\diagdown j$ with $i,j\in I$. Then by (C4),
$a\diagdown i\oplus i\sim b\diagdown j\oplus j$, hence $a\sim_I b$ implies $a\sim b$. Conversely, assume that $a\sim b$. Then by (CR), there is $c\in P$ with $a\diagdown c \sim b\diagdown c \sim 0$, and
with $i=a\diagdown c\in I$, $j=b\diagdown c\in I$ we have $i\diagup a=c=j\diagup b$, whence $a\sim_I b$.

\end{proof}

From here on in this paper, \emph{we shall be considering the situation in which
$P$ is a GPEA, $\gamma\colon P\to P$ is a unitizing GPEA-automorphism on $P$, and
the PEA $U$ is the $\gamma$-unitization of $P$ as per Theorem
\ref{th:gammaunitization}. Thus, $P\sp{\eta}$ is disjoint from $P$, $U=P\cup
P\sp{\eta}$, and $\eta\colon P\to P\sp{\eta}$ is a bijection. Moreover, for
all $a\in P$, $\eta a=a\sp{\sim}$, $\gamma a=a\sp{--}$, and $\gamma\sp{-1}a=
a\sp{\sim\sim}$. In what follows, we use $a,b,c,d,$ and $e$, with or without
subscripts, to denote elements of $P$.}

\begin{theorem}\label{th:Pnormalriesz} Let $P$ be a GPEA, and let $U$ be its
$\gamma$-unitization. Then $P$ is a normal Riesz ideal in $U$ iff $P$ is upward
directed.
\end{theorem}

\begin{proof}
As mentioned earlier, by Theorem \cite[Theorem 3.3 (vii)]{FoPuUnit} $P$ is a maximal
proper normal ideal in $U$. Let $P$ be a Riesz ideal in $U$ and let $a,b\in P$.
Then $b\leq 1\sb{U}=a+\eta a$ implies that there are $j,k\in P$ such that $j\leq a,
k\leq\eta a$ and $b\leq j\oplus k$. But then $a\oplus k$ is an upper bound of both
$a$ and $b$, hence $P$ is upward directed.

Conversely, assume that $P$ is upward directed. By Theorem \ref{th:upwardR1}, it
suffices to check condition R1. If $i,a,b\in P$, $i\leq a\oplus b$, this is obvious.
Let $i\leq a+ \eta b=\eta(b\diagdown a)$. Then $(b\diagdown a)\oplus i$ exists in
$P$, and we choose $d\in P$ such that $(b\diagdown a)\oplus i\leq d$, $b\leq d$. Then
$i\leq (b\diagdown a)\diagup d=a\oplus(b\diagup d)$, where $b\diagup d\leq \eta b$.

Now let $i\leq \eta a + b$. We have $\eta a+ b=b+\eta a_1$ for some $a_1\in P$. From
the previous part of this proof we have that $i\leq b\oplus c$ where $c\leq \eta a_1$.
Let $c+u=\eta a_1$. Then $\eta a+b=b+\eta a_1=b+c+u=c_1+b+u=u_1+c_1+b$. From this,
$u_1+c_1=\eta a$, hence $c_1\leq \eta a$, and $i\leq b\oplus c=c_1\oplus b$.
\end{proof}

\begin{definition}\label{de:gammaclosed} Let $P$ be a GPEA with a unitizing
automorphism $\gamma$. An ideal $I$ in $P$ is called \emph{$\gamma$-closed}
(or simply a $\gamma$-ideal) iff for all $a\in P$, $a\in I$ $\Leftrightarrow$
$\gamma a\in I$, or equivalently iff $I=\gamma I=\{\gamma i:i\in I\}$. A congruence
$\sim$ on $P$ is called a \emph{$\gamma$-congruence} iff $a\sim b\Leftrightarrow
\gamma a\sim\gamma b$.
\end{definition}

\begin{theorem}\label{th:restrictid}  Let $P$ be a GPEA, $\gamma$ a unitizing
automorphism of $P$, and  $U$ the $\gamma$-unitization of $P$. If $I$ is a normal
Riesz ideal in $U$, then its restriction $I\cap P$ to $P$ is a $\gamma$-closed
normal Riesz ideal in $P$.
\end{theorem}

\begin{proof} Taking into account Lemma \ref{le:normalidealprops2} (i) and the fact
that for any $a\in P$, $\gamma a=a^{--}\in P$, the proof is straightforward.
\end{proof}

Now we shall consider the question of extending a congruence on the GPEA $P$ to
a congruence on its $\gamma$-unitization $U$.

\begin{definition} \label{df:extendsim}
Let $U$ be the $\gamma$-unitization of a GPEA $P$ and let $\sim$ be a weak congruence
on $P$. Since $U=P\cup P\sp{\eta}$, $P\sp{\eta}=U\setminus P$, and $\eta\colon P
\to P\sp{\eta}$ is a bijection, we can, and do, define a binary relation $\sim\sp{*}$
on $U$ by
$$
a\sim\sp{*}b\text{\ iff\ }a\sim b\text{\ and\ }\eta a\sim\sp{*}\eta b\text
{\ iff\ }a\sim b\text{\ for all\ }a,b\in P.
$$
If $a\in P$ and $x\in P\sp{\eta}=U\setminus P$, we understand that $a\not\sim\sp{*}x$.
\end{definition}

\begin{lemma} \label{le:gammacongprops}
Let $U$ be the $\gamma$-unitization of the GPEA $P$, let $\sim$ be a
$\gamma$-congruence on $P$, and let $a,b\in P$. Then{\rm: (i)} $a\sp{-}
\sim\sp{*}b\sp{-}$ iff $a\sim b$. {\rm(ii)} $a\sim b\sp{--}$ iff
$a\sp{\sim\sim}\sim b$ iff $a\sp{\sim}\sim\sp{*}b\sp{-}$.
\end{lemma}

\begin{proof} Assume the hypotheses of the lemma.
(i) $a\sp{-}\sim\sp{*}b\sp{-}$ iff $a\sp{--\sim}\sim\sp{*}b\sp{--\sim}$
iff $(\gamma a)\sp{\sim}\sim\sp{*}(\gamma b)\sp{\sim}$ iff $(\gamma a)
\sim(\gamma b)$ iff $a\sim b$.

(ii) $a\sim b\sp{--}$ iff $a\sp{\sim\sim--}\sim b\sp{--}$ iff $\gamma
(a\sp{\sim\sim})\sim\gamma b$ iff $a\sp{\sim\sim}\sim b$. Also, by (i),
$a\sp{\sim\sim}\sim b$ iff $a\sp{\sim}=a\sp{\sim\sim-}\sim\sp{*}b\sp{-}$.
\end{proof}

\begin{theorem}\label{ThCongr}
Let $U$ be the $\gamma$-unitization of the GPEA $P$ and let $\sim$ be a congruence
on $P$. Then $\sim^*$ is a congruence on $U$ iff $\sim$ is a $\gamma$-congruence
that satisfies \rm{(C4)} and \rm{(C5$\sp{\prime}$)}.
\end{theorem}

\begin{proof} Assume that $\sim$ is a $\gamma$-congruence on $P$ that satisfies (C4) and
(C5$\sp{\prime}$). Evidently, the restriction of $\sim^*$ to $P$ as well as
to $P\sp{\eta}$ is an equivalence relation; hence $\sim^*$ is an equivalence
relation on $U=P\cup P\sp{\eta}$, and we have (C1).

To prove (C2), we concentrate only on the two nontrivial cases corresponding
to parts (2) and (3) of Theorem \ref{th:gammaunitization}. Thus let $a\sim^*
a_1$, $b^{\sim}\sim^*b_1^{\sim}$ and suppose that $a+b^{\sim}$ and $a_1+
b_1^{\sim}$ exist. Then $a\leq b$ and $a_1\leq b_1$, so we have $(b\diagdown
a)\oplus a=b\sim b_1=(b_1\diagdown a_1)\oplus a_1$.  Therefore, by (C4),
$b\diagdown a\sim b_1\diagdown a_1$, whence $a+b^{\sim}=(b\diagdown a)
^{\sim}\sim^*(b_1\diagdown a_1)^{\sim}=a_1+b_1^{\sim}$ by Theorem
\ref{th:gammaunitization} (2).

For the remaining case, let $a^{\sim}\sim^* a^{\sim}_1$, $b\sim^* b_1$ and
suppose that $a^{\sim}+b$ and $a_1^{\sim}+b_1$ exist. Then $\gamma b=b^{--}
\leq a$, so we have $\gamma b\oplus(\gamma b\diagup a)=a\sim a_1=\gamma b_1
\oplus(\gamma b_1\diagup a_1)$. Since $b\sim b_1$, it follows that $\gamma b
\sim\gamma b_1$ and (C4) implies that $\gamma b\diagup a\sim\gamma b_1
\diagup a_1$. But then, $(\gamma b\diagup a)^{\sim}\sim^*(\gamma b_1
\diagup a_1)^{\sim}$, and we infer from Theorem \ref{th:gammaunitization}
(3) that $a^{\sim}+b\sim^*a_1^{\sim}\oplus b_1$. Thus $\sim\sp{*}$ satisfies
(C2).

To prove (C3), we again concentrate only on the nontrivial cases corresponding
to parts (2) and (3) of Theorem \ref{th:gammaunitization}. Thus, on the one
hand, assume that $a+b\sp{\sim}$ exists and let $a\sb{1}\in P$ with $a\sb{1}
\sim a$. Then $a\leq b$ and $b=(b\diagdown a)\oplus a$. Therefore by (C3) (in $P$)
there exists $c\in P$ such that $c\sim b\diagdown a$ and $c\oplus a\sb{1}$ exists.
By (C2) (in $P$), $b\sb{1}:=c\oplus a\sb{1}\sim b$, whence $b\sb{1}\sp{\sim}
\sim\sp{*}b\sp{\sim}$ and, since $a\sb{1}\leq b\sb{1}$, it follows that $a\sb{1}
+b\sb{1}\sp{\sim}$ exists.

Continuing to assume that $a+b\sp{\sim}$ exists, we let $b\sb{2}\in P$ with
$b\sb{2}\sp{\sim}\sim\sp{*}b\sp{\sim}$, whence $b\sb{2}\sim b=(b\diagdown a)
\oplus a$. But then, by (C5$\sp{\prime}$), there exist $c, a\sb{2}\in P$ such
that $b\sb{2}=c\oplus a\sb{2}$ where $c\sim b\diagdown a$ and $a\sb{2}\sim a$.
Since $a\sb{2}\leq b\sb{2}$, it follows that $a\sb{2}+b\sb{2}\sp{\sim}$ exists.

On the other hand, assume that $a\sp{\sim}+b$ exists and let $a\sb{1}\in P$
with $a\sb{1}\sp{\sim}\sim\sp{*}a\sp{\sim}$. Then $\gamma b\leq a$ and $a\sb{1}
\sim a=\gamma b\oplus(\gamma b\diagup a)$. But then, by (C5$\sp{\prime}$), there
exist $c,d\in P$ such that $a\sb{1}=c\oplus d$ where $c\sim\gamma b=b\sp{--}$
and $d\sim \gamma b\diagup a$. But then, by Lemma \ref{le:gammacongprops} (ii),
$b\sb{1}:=c\sp{\sim\sim}\sim b$; moreover, since $c\leq a\sb{1}$, we have $a\sb{1}
\sp{\sim}\leq c\sp{\sim}$, so $a\sb{1}\sp{\sim}+c\sp{\sim\sim}=a\sb{1}\sp{\sim}
+b\sb{1}$ exists.

Continuing to assume that $a\sp{\sim}+b$ exists, we let $b\sb{2}\in P$ with
$b\sb{2}\sim\sp{*}b$, i.e., $b\sb{2}\sim b$. Then $\gamma b\sb{2}\sim
\gamma b$, and since $\gamma b\oplus(\gamma b\diagup a)=a$ exists, (C3) implies
that there exists $c\in P$ such that $c\sim\gamma b\diagup a$ and $\gamma b\sb{2}
\oplus c$ exists. But then, by (C2), $a\sb{2}:=\gamma b\sb{2}\oplus c\sim
\gamma b\oplus(\gamma b\diagup a)=a$, so $a\sb{2}\sp{\sim}\sim\sp{*}a\sp{\sim}$.
Since $b\sb{2}\sp{--}=\gamma b\sb{2}\leq a\sb{2}$, we have $a\sb{2}\sp{\sim}\leq
b\sb{2}\sp{--\sim}=b\sb{2}\sp{-}$, whence $\gamma b\sb{2}=b\sb{2}\sp{--}\leq a
\sb{2}\sp{\sim -}=a$, and it follows that $a\sb{2}\sp{\sim}+b\sb{2}$ exists.
Thus $\sim\sp{*}$ satisfies (C3), so it is a congruence on $U$.

Conversely, if $\sim^*$ is a congruence on the PEA $U$, then by Theorem
\ref{th:congrpea}, $\sim^*$ satisfies (C4$\sp{\prime}$), (C4), and
(C5$\sp{\prime}$), and thus, of course, $\sim$ satisfies (C4) and
(C5$\sp{\prime}$) on $P$. Also, since $\sim\sp{*}$ satisfies (C4$\sp{\prime}$),
it follows that $\sim$ is a $\gamma$-congruence on $P$.
\end{proof}

\begin{proposition} \label{pr:equivalences}
\emph{Let $I$ be a normal {\rm(R1)}-$\gamma$-ideal in the GPEA $P$, put
$\sim :=\sim\sb{I}$, and let $U$ be the $\gamma$-unitization of $P$. Then
the following conditions are mutually equivalent{\rm: (i)} $\sim^*$ satisfies
{\rm(C3)} on $U$. {\rm(ii)} $I$ is a normal Riesz $\gamma$-ideal in $P$.
{\rm(iii)} $\sim$ is a $\gamma$-congruence on $P$ that satisfies {\rm(C4)}
and {\rm(C5$\sp{\prime}$)}. {\rm(iv)} $\sim\sp{*}$ is a congruence on $U$
that satisfies {\rm(C4)}, {\rm(C4$\sp{\prime}$)}, {\rm(C5$\sp{\prime}$)},
and {\rm(C5)}.}
\end{proposition}

\begin{proof} By Theorem \ref{th:R1congr}, $\sim=\sim\sb{I}$ is a weak
congruence on $P$ satisfying (C5$\sp{\prime}$).

\smallskip

\noindent (i) $\Rightarrow$ (ii). Suppose that $\sim^*$ satisfies (C3) on
$U$. Then also $\sim$ satisfies (C3) on $P$ and so it is a congruence on
$P$. Assume that $i\in I$ with $i\leq a\in P$ and that $(a\diagdown i)
\oplus b$ exists for some $b\in P$. Then we have $a\diagdown i\sim a$ and
so by (C3) there is $b\sb{1}\in P$ such that $a\oplus b_1$ exists, whence
$b\sb{1}\leq a\sp{\sim}$, and $b\sim b\sb{1}$. Since $I$ is normal and $b
\sim b\sb{1}$, there exist $j,k\in I$ with $j\diagup b=k\diagup b\sb{1}$.
Thus we have $j\diagup b=k\diagup b\sb{1}\leq k\oplus(k\diagup b\sb{1})=
b\sb{1}\leq a\sp{\sim}$, and it follows that $a\oplus j\diagup b$ exists.
This proves part (i) of condition (R2) in Definition \ref{de: riesz}, a
similar argument proves part (ii), whence $I$ is also an (R2)-ideal, and
we have (ii).

\smallskip

\noindent (ii) $\Rightarrow$ (iii). Suppose that $I$ is a normal Riesz
$\gamma$-ideal in $P$. Then by Theorems \ref{th:R1congr} and \ref
{th:rieszcongr}, $\sim=\sim\sb{I}$ is a $\gamma$-congruence on $P$
satisfying (C4) and (C5$\sp{\prime}$).

\smallskip

\noindent (iii) $\Rightarrow$ (iv). If $\sim$ is a $\gamma$-congruence
on $P$ satisfying (C4) and (C5$\sp{\prime}$), then by Theorem \ref{ThCongr},
$\sim\sp{*}$ is a congruence on $U$. By Theorem \ref{th:congrpea},
$\sim\sp{*}$ satisfies (C4$\sp{\prime}$), (C4), (C5$\sp{\prime}$), and
(C5).

\smallskip

\noindent (iv) $\Rightarrow$ (i). A congruence $\sim\sp{*}$ on $U$ satisfies
(C3) on $U$.
\end{proof}

\begin{theorem} \label{th:GCRcondition}
Let $I$ be a normal Riesz $\gamma$-ideal in the GPEA $P$, put $\sim:=\sim\sb{I}$,
and let $U$ be the $\gamma$-unitization of $P$. Then $\sim^*$ is a Riesz
congruence on  $U$ iff $\sim$ satisfies the following condition on $P${\rm:}
\begin{enumerate}
\item[ ]
 \begin{enumerate}
\item[{\rm(GCR)}] $a\sim b\,\Rightarrow\, \exists\,i,j\in I: i\oplus a=j\oplus b$,
or equivalently $\exists\,k,\ell\in I: a\oplus k=b\oplus \ell$.
 \end{enumerate}
\end{enumerate}
\end{theorem}

\begin{proof} Assume the hypotheses of the theorem. By Proposition
\ref{pr:equivalences}, $\sim$ is a $\gamma$-congruence on $P$, $\sim\sp{*}$ is a
congruence on $U$, and both satisfy (C4) and (C5$\sp{\prime}$). Also, by Theorem
\ref{th:R1congr}, $I=\{a\in P:a\sim 0\}$. Thus, if $a\in P$, $i\in I$, and $i
\oplus a$ exists, then by (C2), $a=0\oplus a\sim i\oplus a$, and therefore
$a\sp{\sim}\sim\sp{*}(i\oplus a)\sp{\sim}$.

Assume that $\sim$ satisfies condition (GCR). Consider the following condition
on $\sim\sp{*}$: For all $u,v\in U$,

\smallskip

(CR$\sp{\prime}$) $\;\; u\sim\sp{*}v\,\Rightarrow\,\exists\, s,t\in U: u+s\sim\sp{*}
 v+s\sim\sp{*}1$, and $t+u\sim\sp{*}t+v\sim\sp{*}1.$

\smallskip

According to \cite[Proposition 3.1]{XL}, on the PEA $U$, condition (CR$\sp{\prime}$)
is equivalent to condition (CR). We shall show that $\sim^*$ satisfies condition
(CR$\sp{\prime}$).

Let $a,b\in P$. On the one hand, suppose that $a\sim^*b$. Then $a\sim b$, whence
by (GCR), there exist $i,j,k,\ell\in I$ with $i\oplus a=j\oplus b$ and $a\oplus k
=b\oplus\ell$. Put $s:=(i\oplus a)\sp{\sim}=(j\oplus b)\sp{\sim}$ and $t:=(a
\oplus k)\sp{-}=(b\oplus\ell)\sp{-}$. Thus, $t=(\gamma(a\oplus k))\sp{\sim}=
(\gamma(b\oplus\ell))\sp{\sim}$.  Since $i\oplus a\sim a$, we have $s\sim
\sp{*}a\sp{\sim}\sim\sp{*}b\sp{\sim}$ and similarly, $t\sim\sp{*}a\sp{-}\sim
\sp{*}b\sp{-}$. Also, $a\leq i\oplus a$, and it follows that $a+s$ exists and
$a+s\sim\sp{*}a+a\sp{\sim}=1$. Similar computations show that $b+s\sim\sp{*}1$,
$t+a\sim\sp{*}1$, and $t+b\sim\sp{*}1$.

On the other hand, suppose that $a\sp{\sim}\sim\sp{*}b\sp{\sim}$. Then $a\sim
\sb{I}b$, whence there exist $i,j\in I$ with $a\diagdown i=b\diagdown j$. Put
$s:=(a\diagdown i)\sp{\sim\sim}=(b\diagdown j)\sp{\sim\sim}$ and $t:=
a\diagdown i=b\diagdown j$. We have $t=a\diagdown i\sim a$, whence $s\sim a
\sp{\sim\sim}$. Also $\gamma s=\gamma((a\diagdown i)\sp{\sim\sim})=a\diagdown i
\leq a$, and it follows that $a\sp{\sim}+s$ exists and $a\sp{\sim}+s\sim\sp{*}a
\sp{\sim}+a\sp{\sim\sim}=1$. Similar computations show that $b\sp{\sim}+s\sim
\sp{*}1$, $t+a\sp{\sim}\sim\sp{*}1$, and $t+b\sp{\sim}\sim\sp{*}1$. Thus, $\sim
\sp{*}$ satisfies (CR$\sp{\prime}$), whence it satisfies (CR), and consequently
it is a Riesz congruence on $U$.

Conversely, let $\sim^*$ be a Riesz congruence on $U$ and suppose that
$a,b\in P$ with $a\sim b$. Then by (CR), there exists $d\in U$ with $d
\diagdown a\sim\sp{*}d\diagdown b\sim^* 0$. Thus $(d\diagdown a)+a=d$,
and since $0+a=a$ with $0\sim\sp{*}d\diagdown a$ and $a\sim\sp{*}a$,
it follows that $d\sim\sp{*}a\in P$; hence $d\in P$, $d\sim a\sim b$,
and $d\diagdown a\sim\sb{I}d\diagdown b$ in $P$.  Therefore, there
exist $i,j\in I$ such that $d\diagdown(i\oplus a)=(d\diagdown a)
\diagdown i=(d\diagdown b)\diagdown j=d\diagdown(j\oplus b)$. Consequently,
$i\oplus a=j\oplus b$ and (GCR) holds.
\end{proof}

\begin{theorem} \label{th:normRgammaideal&GCR}
If $I$ is a normal Riesz $\gamma$-ideal in the GPEA $P$, then $I$ is a
normal Riesz ideal in the $\gamma$-unitization $U$ of $P$ iff $\sim\sb{I}$
satisfies the {\rm(GCR)} condition on $P$.
\end{theorem}

\begin{proof}
If $I$ is a normal Riesz $\gamma$-ideal in $P$ and $\sim\sb{I}$
satisfies the (GCR) condition, then by Theorem \ref{th:GCRcondition}
$\sim\sb{I}\sp{*}$ is a Riesz congruence on $U$. Since $I$ is the zero class
of this congruence, it is a normal Riesz ideal in $U$ (Theorem
\ref{th:newtheor} (ii)).

Conversely, suppose that $I$ is a normal Riesz ideal in $U$ and let $\sim
\sb{U,I}$ be the relation induced on $U$ by $I$ as per Definition
\ref{df:simsbI}. Clearly, the restriction of $\sim\sb{U,I}$ to $P$ coincides
with $\sim\sb{I}$. Also, since $U$ is upward directed, $\sim\sb{U,I}$ is a
Riesz congruence on $U$ by Theorem \ref{th:upward}. Thus, by Theorem
\ref{th:congrpea}, $\sim\sb{U,I}$ satisfies (C4) and (C5$\sp{\prime}$)
on $U$, whence by Theorem \ref{th:newtheor} (i), all equivalence classes
in $U$ modulo $\sim\sb{U,I}$ are upward directed. To prove that $\sim\sb{I}$
satisfies GCR, assume that $a,b\in P$ with $a\sim\sb{I}b$. Then $a\sim
\sb{U,I}b$, so there exists $r\in U$ such that $a,b\leq r$ and $a\sim\sb{U,I}b
\sim\sb{U,I}r$. As $a,b\leq r$, there exist $p,q\in U$ with $a+p=b+q=r$.
Moreover, as $a\sim\sb{U,I}r$, there exist $i,j\in I$ such that $i\leq a$,
$j\leq r=a+p$, and $a\diagdown i=r\diagdown j=(a+p)\diagdown j$. Therefore,
$(a\diagdown i)+j=a+p=(a\diagdown i)+i+p$, whence $j=i+p$. But then, $p
\leq j\in I$, so $p\in I$. A similar argument shows that $q\in I$, and
therefore GCR holds.
\end{proof}

We remark that, in the proof of Theorem \ref{th:normRgammaideal&GCR},
$\sim\sb{U,I}=\sim\sb{I}\sp{*}$. To prove this, it will be sufficient
to show that, for all $a,b\in P$, $a\sim\sb{I}b\Leftrightarrow
a\sp{\sim}\sim\sb{U,I}b\sp{\sim}$ and that $a\sim\sb{U,I}b\sp{\sim}$
cannot hold. But by Theorem \ref{th:congrpea}, $a\sim\sb{I}b
\Rightarrow a\sim\sb{U,I}b\Rightarrow a\sp{\sim}\sim\sb{U,I}b\sp{\sim}$.
Conversely, by the same theorem, $a\sp{\sim}\sim\sb{U,I}b\sp{\sim}
\Rightarrow a=a\sp{\sim -}\sim\sb{U,I}b\sp{\sim -}=b\Rightarrow a\sim
\sb{I}b$. Moreover, if $a\sim\sb{U,I}b\sp{\sim}$, then there exist $i,j
\in I$ such that $a\diagdown i=b\sp{\sim}\diagdown j$, which yields
$(a\diagdown i)\oplus j=a\diagdown i+j=b\sp{\sim}\not\in P$, contradicting
the definition of $U$.

\begin{corollary}\label{co:corol}
In an upward directed GPEA $P$, a normal Riesz $\gamma$-ideal $I$ in $P$ is also
a normal Riesz ideal in the $\gamma$-unitization $U$ of $P$.
\end{corollary}

\begin{proof}
By the previous theorem, it is sufficient to show that $\sim_I$ satisfies (GCR).
So assume that $a\sim_I b$. Since $P$ is upward directed, there exists $c\in P$
with $a,b\leq c$, and by (C4), $c\diagdown a\sim_I c\diagdown b$. Thus there are
$i,j\in I$ such that $c\diagdown (i\oplus a)=(c\diagdown a)\diagdown i=
(c\diagdown b)\diagdown j=c\diagdown (j\oplus b)$; hence $i\oplus a=j\oplus b$.
\end{proof}

\begin{theorem}
Let $\sim$ be a $\gamma$-congruence on a GPEA $P$ which satisfies conditions
{\rm(C4)} and {\rm(C5$\sp{\prime}$)}. Let $\sim^*$ be the extension of $\sim$ in
the $\gamma$-unitization $U$ of $P$. Define $\widetilde{\gamma}:P/{\sim} \to P/{\sim}$
by $\widetilde{\gamma}[a]=[\gamma a]$. Then $\widetilde{\gamma}$ is a unitizing
automorphism in $P/{\sim}$, and $U/\sim^*$ is the  $\widetilde{\gamma}$-unitization
of $P/\sim$. In particular, if $I$ is a normal Riesz $\gamma$-ideal on $P$ such that
$\sim_I$ satisfies {\rm(GCR)}, then the $\widetilde{\gamma}$-unitization of $P/I$ is
$U/I$.
\end{theorem}

\begin{proof}
By Theorem \ref{ThCongr}, $\sim^*$ is a congruence on $U$ and so $U/\sim^*$ is a
PEA. Clearly, if $s,t\in U$ and $s\in [t]$ (i.e. $s\sim\sp{*}t$), then either
$s,t\in P$ or else $s,t\in U\setminus P=:\eta P$; hence $[a]\cap[\eta b]=
\emptyset$ for all $a,b\in P$. Consequently, $U/\sim\sp{*}=P/\sim\sp{*}\cup\,
\eta P/\sim\sp{*}$, where the sets on the right are disjoint and obviously of
the same cardinality.  Thus we may put $\eta[a]:=[\eta a]$.

We have $[a]\oplus [b]\in P/{\sim}$ for all $a,b\in P$ and for $a,b\in \eta P$,
$[a]\oplus [b]$ does not exist. Also $[1]=[a]^-\oplus [a]=[a]\oplus [a]^{\sim}$
for any $a\in P$ and thus $[1]\not\in P/\sim$. Therefore $U/\sim^*$ satisfies
(U1)-(U3) and it is a unitization of $P/\sim$ according to Definition
\ref{de:unitize}.

As $\sim$ is a $\gamma$-congruence, $\widetilde{\gamma}[a]=[\gamma a], a\in P$, is well
defined. If $[a]\oplus [b]$ is defined, there are $a_1\sim a, b_1\sim b$ such that
$[a]\oplus [b]=[a_1\oplus b_1]$. Then $\widetilde{\gamma}([a]\oplus [b])=[\gamma(a_1
\oplus b_1)]=[\gamma a_1]\oplus [\gamma b_1]=\widetilde{\gamma}[a]\oplus\widetilde{\gamma}
[b]$, hence $\widetilde{\gamma}$ is additive. Moreover, $\widetilde{\gamma}[a]=\widetilde
{\gamma}[b]$ implies $\gamma a\sim \gamma b$, which in turn implies $a\sim b$, so that
$[a]=[b]$. This shows that $\widetilde{\gamma}$ is injective. To prove that $\widetilde
{\gamma}$ is surjective, observe that for all $a\in P$, $a=\gamma(\gamma^{-1} a)$,
hence $[a]=\widetilde{\gamma}[\gamma ^{-1} a]$. This proves that $\widetilde{\gamma}$ is
an automorphism of $P/{\sim}$. In addition, if $[a]\oplus [b]$ exists, then there are
$a_1\sim a$ and $b_1\sim b$ such that $a_1\oplus b_1$ is defined in $P$. Now
$a_1\oplus b_1$ is defined iff $\gamma b_1\oplus a_1$ is defined in $P$, whence
$[a]\oplus [b]$ exists iff $\widetilde{\gamma}[b]\oplus [a]$ exists. This proves that
$\widetilde{\gamma}$ is a unitizing automorphism of $P/{\sim}$.

In $U/{\sim^*}$ we have $[a]+ \eta [b]$ is defined iff there are $a_1,b_1$, $a_1
\sim a, b_1\sim b$ such that $a_1+ \eta b_1=\eta(b_1\diagdown a_1)$ is defined
and then $[a]+ \eta[b] = \eta[b_1\diagdown a_1]=\eta([b]\diagdown [a])$. Similarly,
$\eta[a]+[b]$ is defined iff there are $a_1\sim a, b_1\sim b$ such that $\eta a_1+
b_1=\eta(\gamma b_1\diagup a_1)$ is defined, and then $\eta[a]+[b]=\eta[\gamma b_1
\diagup a_1]=\eta(\widetilde{\gamma}[b]\diagup [a])$. This shows that $U/{\sim^*}$ is a
$\widetilde{\gamma}$-unitization of $P/{\sim}$.
\end{proof}

\begin{example}{\rm \cite[Example 2.3]{DXY}}
Let $\mathbb{Z}$ be the group of integers and $G=\mathbb{Z}\times\mathbb{Z}\times
\mathbb{Z}$. For any two elements of $G$ we define:
\[
(a,b,c)+(x,y,z)=\left \{ \begin{array}
                   {r@{\quad:\quad}l}
                    (a+x,b+y,c+z) & x \mbox{ is even} \\ (a+x,c+y,b+z) & x \mbox{ is odd}
                    \end{array} \right.
                                        \]
and define $(a,b,c)\leq (x,y,z)$ iff $a<x$ or $a=x, b\leq y, c\leq z$. Then $(G,+,\leq)$
is a lattice-ordered group with strong unit $u=(1,0,0)$ and the interval $E:=
G[(0,0,0),(1,0,0)]$ is a PEA.

The elements between $(0,0,0)$ and $(1,0,0)$ are of two kinds---$(0,a,b)$ and
$(1,c,d)$---where $a,b\geq 0$ and $c,d\leq 0$. The two sets are of the same cardinality
and $P:=\{(0,a,b)\in G: a,b\geq 0\}$ is a normal Riesz ideal of $E$. We also have
$(0,b,c)^{\sim}=(1,-c,-b)$ and $(0,b,c)^-=(1,-b,-c)$, thus $\gamma (0,a,b)=(0,a,b)^{--}=
(0,b,a)$. With this $\gamma$ on $P$, which is a weakly commutative GPEA, we have a GPEA
with a unitizing automorphism and its $\gamma$-unitization is exactly the $G[(0,0,0),
(1,0,0)]$ that we started with.

We may also consider the identity as a unitizing automorphism on $P$ and then we obtain
a unitization $G[(0,0,0),(1,0,0)]$ but with the $+$ operation on $G$ defined by
$$
(a,b,c)+(x,y,z)=(a+x,b+y,c+z)
$$
for all $a,b,c,x,y,z\in\mathbb{Z}$.
\end{example}

\section{The Riesz decomposition property}

\begin{definition}
We say, that a GPEA $P$ satisfies
\begin{itemize}
\item[(i)] the \emph{Riesz decomposition property} (\emph{RDP} for short), iff
for all $a,b,c,d\in P$ such that $a\oplus b=c\oplus d$, there are four elements
$e_{11}, e_{12}, e_{21}, e_{22}\in P$ such that $a=e_{11}\oplus e_{12}$, $b=e_{21}
\oplus e_{22}$, $c=e_{11}\oplus e_{21}$ and $d=e_{12}\oplus e_{22}$. We shall
denote these four decompositions by the following table:

\begin{tabular}{c|cc}
& c & d\\
\hline
a & $e_{11}$ & $e_{12}$\\
b & $e_{21}$ & $e_{22}$\\
\end{tabular}

\vspace{1em}
\item[(ii)] RDP$_1$, if for the decomposition in (i) it moreover holds: if $f,g
\in P$ are such that $f\leq e_{12}$, $g\leq e_{21}$, then $f$ and $g$
commute ($f\oplus g=g\oplus f$).
\item[(iii)] RDP$_2$, if for the decomposition in (i) it moreover holds that
 $e_{12}\wedge e_{21}=0$.
\item[(iv)] RDP$_0$, if for $a,b,c\in P$ such that $a\leq b\oplus c$, there
exists $b_1,c_1\in P$ such that $b_1\leq b$, $c_1\leq c$ and $a=b_1\oplus c_1$.
\end{itemize}
\end{definition}

From the next example we see that, even in the commutative case, if $P$ has RDP,
its unitization $U$ need not have it.

\begin{example}{\rm \cite[Example 4.1]{RiMa}}
Consider the GEA $P=\{0,a,b,c,a\oplus c,\linebreak b\oplus c\}$ {\rm(}see Fig. 1
below{\rm)} and let the unitizing automorphism $\gamma$ be the identity {\rm(}thus
the unitization $U$ of $P$ is the ``classical" unitization ${\widehat P}${\rm)}.
Then $P$ has RDP, but $U$ does not. Indeed, for example, in $U$, $(b\oplus c)
\sp{\perp}=b\sp{\perp}\ominus c$ cannot be decomposed except for $b\sp{\perp}
\ominus c=(b\sp{\perp}\ominus c)\oplus 0$. But $(b\sp{\perp}\ominus c)\oplus b
=c\sp{\perp}=(a\sp{\perp}\ominus c)\oplus a$ and therefore, for these elements,
there is no decomposition in the sense of RDP.
\newpage
\vspace{2em}

{\footnotesize
\begin{picture}(60,0)(60,0)
\unitlength=4pt

\put(60,-20){\circle*{1,2}} \put(57,-20){\makebox(0,0)[l]{$0$}}
\put(50,-10){\circle*{1,2}} \put(47,-10){\makebox(0,0)[l]{$a$}}
\put(60,-20){\line(-1,1){10}}
\put(55,-10){\circle*{1,2}} \put(57,-10){\makebox(0,0)[l]{$c$}}
\put(60,-20){\line(-1,2){5}}
\put(65,-10){\circle*{1,2}} \put(67,-10){\makebox(0,0)[l]{$b$}}
\put(60,-20){\line(1,2){5}}
\put(52.5,0){\circle*{1,2}} \put(50,0){\makebox(0,0)[r]{$a\oplus c$}}
\put(50,-10){\line(1,4){2.5}}
\put(55,-10){\line(-1,4){2.5}}
\put(60,0){\circle*{1,2}} \put(62,0){\makebox(0,0)[l]{$b\oplus c$}}
\put(55,-10){\line(1,2){5}}
\put(65,-10){\line(-1,2){5}}
\put(55,-28){\normalsize Fig. 1}
\end{picture}
}

\vspace{10em}
\end{example}

In the following theorem, we show that if a GPEA $P$ has a total operation
$\oplus$, then $P$ has any of the RDP properties iff its $\gamma$-unitization
(with respect to any unitizing automorphism $\gamma$) has the corresponding RDP
property. We do not know whether the above condition is also necessary.

\begin{theorem}
Let $P$ be a total GPEA and let $U$ be the unitization of $P$ by a unitizing
GPEA-automorphism $\gamma$. Then $P$ has RDP {\rm(}RDP$_0$, RDP$_1$, RDP$_2$,
respectively{\rm)} iff $U$ has RDP {\rm(}RDP$_0$, RDP$_1$, RDP$_2$,
respectively{\rm)}.
\end{theorem}

\begin{proof}Notice first, that, since $a\oplus b$ exists for all $a,b\in P$,
then of course, $P$ is upward directed. Also, since $a+b=a\oplus b$ exists,
then $b\leq a\sp{\sim}=\eta a$ for all $a,b\in P$.

It is clear, that if $U$ has RDP (RDP$_1$, RDP$_2$), then $P$ has RDP
(RDP$_1$, RDP$_2$), because if the decomposition of $a\oplus b=c\oplus d$
is:\\

\begin{tabular}{c|cc}
& c & d\\
\hline
a & $e_{11}$ & $e_{12}$\\
b & $e_{21}$ & $e_{22}$\\
\end{tabular}

\vspace{1em}
\noindent
then $e_{11}, e_{12}\leq a$ and $e_{21}, e_{22}\leq b$ and so if
$a,b,c,d\in P$, then $e_{11}, e_{12}, e_{21},\\ e_{22}\in P$, because $P$
is an ideal in $U$. Similarly, if $U$ has RDP$_0$, so does $P$.

Now assume that RDP holds in $P$. Then we have to consider three cases:
\begin{itemize}

\item[(a)] $\eta a+b=\eta c+d$ (where $a,b,c,d\in P$). Then
$\eta (\gamma b\diagup a)=\eta (\gamma d\diagup c)$, so $\gamma b\diagup a
=\gamma d\diagup c$. Since $P$ is upward directed, there exists $e\in P$
with $e\geq a,c$ and therefore
$$
e\diagdown(\gamma b\diagup a)=e\diagdown(\gamma d\diagup c), \text{\ i.e.,\ }
(e\diagdown a)\oplus\gamma b=(e\diagdown c)\oplus\gamma d
$$
where all these elements are in $P$, so by RDP we have a decomposition:

\begin{tabular}{c|cc}
& $e\diagdown c$ & $\gamma d$\\
\hline
$e\diagdown a$ & $e_{11}$ & $e_{12}$\\
$\gamma b$ & $e_{21}$ & $e_{22}$\\
\end{tabular}

\vspace{1em}
\noindent
Also, as $(e\diagdown a)+a$ exists, $e\diagdown a\leq a\sp{-}$, whence
$\gamma\sp{-1}(e\diagdown a)\leq\eta a$, and likewise $\gamma\sp{-1}
(e\diagdown c)\leq\eta c$. Therefore
$$\exists x,y\in U,\ x+\gamma\sp{-1}(e\diagdown a)=\eta a\text{\ and\ }
 y+\gamma\sp{-1}(e\diagdown c)=\eta c.$$
But by assumption $\eta a+b=\eta c+d$, and therefore
$$x+\gamma\sp{-1}(e\diagdown a)+b=\eta a+b=\eta c+d=y+\gamma\sp{-1}
 (e\diagdown c)+d.$$
Moreover, from $(e\diagdown a)\oplus\gamma b=(e\diagdown c)\oplus\gamma d$, we
deduce that
$$\gamma\sp{-1}(e\diagdown a)+b=\gamma\sp{-1}(e\diagdown a)\oplus b=
 \gamma\sp{-1}(e\diagdown c)\oplus d=\gamma\sp{-1}(e\diagdown c)+d$$
and it follows that $x=y$. Thus, $\eta a=x+\gamma\sp{-1}(e\diagdown a)$,
$\eta c=x+\gamma\sp{-1}(e\diagdown c)$, and we have

\begin{tabular}{c|cc}
& $\eta c$ & $d$\\
\hline
$\eta a$ & $x+\gamma^{-1} e_{11}$ & $\gamma^{-1} e_{12}$\\
$b$ & $\gamma^{-1} e_{21}$ & $\gamma^{-1} e_{22}$\\
\end{tabular}

\vspace{1em}

\item[(b)] $a+\eta b=c+\eta d$. Then $\eta(b\diagdown a)=\eta
(d\diagdown c)$, so $b\diagdown a=d\diagdown c$. Again, there exists $e
\in P$ with $e\geq b,d$, so that
$$(b\diagdown a)\diagup e=(d\diagdown c)\diagup e, \text{\ i.e.,\ }
a\oplus(b\diagup e)=c\oplus(d\diagup e).$$
Thus, since RDP holds in $P$, we get

\begin{tabular}{c|cc}
& $c$ & $d\diagup e$\\
\hline
$a$ & $e_{11}$ & $e_{12}$\\
$b\diagup e$ & $e_{21}$ & $e_{22}$\\
\end{tabular}

\vspace{1em}
\noindent
As $b\oplus(b\diagup e)$ exists, we have $b\diagup e\leq b\sp{\sim}=
\eta b$ and thus $\eta b=(b\diagup e)+x$ for some $x\in U$. Similarly,
$\eta d=(d\diagup e)+y$ for some $y\in U$. Then
$$a+(b\diagup e)+x=a+\eta b=c+\eta d=c+(d\diagup e)+y$$
and, since $a+(b\diagup e)=c+(d\diagup e)$, it follows that $x=y$.
Consequently, we obtain the decomposition:

\begin{tabular}{c|cc}
& $c$ & $(d\diagup e)+x$\\
\hline
$a$ & $e_{11}$ & $e_{12}$\\
$(b\diagup e)+x$ & $e_{21}$ & $e_{22}+x$\\
\end{tabular}

\vspace{1em}
\noindent
where $(b\diagup e)+x=\eta b$ and $(d\diagup e)+x=\eta d$.

\item[(c)] $a\oplus \eta b=\eta c\oplus d$. As $d\leq b\sp{\sim}=\eta b$
there exists $x\in U$ with $x+d=\eta b$. Then we have $a+x+d=a+\eta b=
\eta c+d$, and therefore $a+x=\eta c$. Thus the required decomposition is:

\begin{tabular}{c|cc}
& $\eta c$ & $d$\\
\hline
$a$ & $a$ & $0$\\
$\eta b$ & $x$ & $d$\\
\end{tabular}
\end{itemize}

From the decompositions obtained in (a), (b), and (c), it is clear
that, if $P$ satisfies RDP$_1$ or RDP$_2$, then so does $U$.

Finally we prove that $P$ has RDP$_0$ iff $U$ has RDP$_0$.
Clearly, if $U$ has RDP$_0$, then $P$ has RDP$_0$.

Assume that $P$ has RDP$_0$. Then there are four cases to check:\\
(i) If $a\leq\eta b\oplus c$, then as $a\leq b\sp{\sim}=\eta b$, it
follows that $a=a\oplus 0$ provides the desired decomposition.\\
(ii) If $a\leq b\oplus\eta c$, then similarly, $a=0\oplus a$ provides
the desired decomposition.\\
(iii) If $\eta a\leq\eta b\oplus c$, then we have $c\leq\eta a$ and
so $\eta a=x\oplus c\leq\eta b\oplus c$ for some $x\in U$ and that
implies $x\leq\eta b$. So $\eta a=x+c$ provides the desired decomposition.\\
(iv) If $\eta a\leq b\oplus\eta c$, then we make use of $b\leq\eta a$ and
so there is again an $x\in U$ such that $\eta a=b+x\leq b\oplus\eta c$, which
implies $x\leq\eta c$, whence $\eta a=b+x$ provides the desired decomposition.

We note that only in cases (c) and (iv) do we use the hypothesis that $P$
is total.
\end{proof}

We recall that a pseudo effect algebra $E$ is a \emph{subdirect product} of a
system of pseudo effect algebras $(E_t:t\in T)$ if there is an injective
homomorphism $h:E\to \prod_{t\in T}E_t$ such that $\pi_t(h(E))=E_t$ for each
$t\in T$, where $\pi_t$ is the $t$-th projection of $\prod_{t\in T} E_t$ onto
$E_t$. In addition, $E$ is \emph{subdirectly irreducible} if whenever $E$ is a
subdirect product of $(E_t:t\in T)$, there is $t_0\in T$ such that $\pi_{t_0}
\circ h$ is an isomorphism of pseudo effect algebras.

In the theory of total algebras, it is well known that an algebra is subdirectly
irreducible iff it has a smallest nontrivial ideal. As pseudo effect algebras are
only partial algebraic structures, some results of universal algebra need not
hold for arbitrary pseudo effect algebras. For example, while for total algebras,
there is one-to-one correspondence between congruences and ideals, in pseudo effect
algebras the relations between ideals and congruences are more complicated.
Therefore, in what follows, instead of studying irreducibility, we study relations
between the smallest nontrivial (i.e. different from $\{ 0\}$) normal Riesz ideals
in GPEAs and in their unitizations.

\begin{theorem}\label{th:smallid} Let $P$ be an upward directed  GPEA and let
$U$ be a $\gamma$-unitiza\-tion of $P$. Then there is a  smallest nontrivial normal
Riesz ideal in $U$ iff there is a smallest normal Riesz $\gamma$-ideal in $P$.
\end{theorem}

\begin{proof} Let $I$ be a smallest nontrivial normal Riesz ideal in $U$. Then
$I\cap P$ is a nontrivial normal Riesz $\gamma$-ideal in $P$ (Theorem
\ref{th:restrictid}). Let $\{ 0\}\neq I_0\subset I\cap P$ be a normal Riesz
$\gamma$-ideal in $P$. By Corollary \ref{co:corol}, $I_0$ is also normal Riesz
ideal in $U$, therefore we must have $I_0=I=I\cap P$.  It follows that $I$ is
also the smallest nontrivial Riesz $\gamma$-ideal in $P$.

Conversely, let $I$ be a smallest nontrivial normal Riesz $\gamma$-ideal in $P$.
By Corollary \ref{co:corol}, $I$ is also a normal Riesz ideal in $U$. If there is
a normal Riesz ideal $\{0\}\neq J$ in $U$ such that $J\subseteq I$, then $\{0\}
\neq J\cap P$ is a normal Riesz $\gamma$-ideal in $P$, and hence $J\cap P=I$. It
follows that $I\subseteq J$, hence $I$ is the smallest nontrivial Riesz ideal in $U$.
\end{proof}

Notice that in an upward directed GPEA $P$ with RDP, every ideal is Riesz ideal.
Similarly as for GEAs \cite[Lemma 3.2]{Dvtat} (see also \cite[Lemma 4.3]{Dvnew}),
it can be proved that a (nontrivial) upward directed  GPEA $P$ with RDP$_1$ is
subdirectly irreducible iff $P$ possesses a smallest non-trivial normal ideal. This
yields the following corollary.

\begin{corollary}\label{co:least} Let $P$ be a GPEA with total operation $\oplus$
satisfying RDP$_1$ and let $U$ be its $\gamma$-unitization. Then $U$ is subdirectly
irreducible iff $P$ has a smallest nontrivial normal $\gamma$-ideal. If $\gamma$ is
the identity, then $U$ is subdirectly irreducible iff $P$ is subdirectly irreducible.
\end{corollary}

\begin{remark}  Assume that $(P;\oplus ,0)$ is an upward directed GPEA, $I$ is a
nonempty indexing set, $\lambda,\rho :I\to I$ are bijections, condition (KCI) holds,
and $(K;+,0_K,1_K)$ is the resulting kite algebra (Theorem \ref{th:kiteisomorphism}).
By Corollary \ref{co:KCI}, $K$ is the $\gamma$-unitization of $P^I$ with the unitizing
automorphism $\gamma (a_i)_{i\in I}=(a_{\rho\circ\lambda^{-1}(i)})_{i\in I}$.
We shall say that $i,j\in I$ are \emph{connected} iff $(\rho\circ\lambda^{-1})^m(i)
=j$, or $(\rho\circ\lambda^{-1})^m(j)=i$  for some integer $m\geq 0$, otherwise $i$
and $j$ are \emph{disconnected}. Let $i_0,j_0\in I$ be disconnected, and let $I_0$ and
$I_1$ be maximal subsets of mutually connected elements in $I$ containing $i_0$ and
$j_0$, respectively. Then clearly, every element in $I_0$ is disconnected with every
element in $I_1$. Let $H_0:=\{ (f_i)_{i\in I}:f_i=0\  \forall i\notin I_0\}$, $H_1:=
\{ (f_j)_{j\in I}: f_j=0 \ \forall j\notin I_1\}$. It is easy to see that $H_0$ and
$H_1$ are normal $\gamma$-ideals in $P^I$. Assume that $K$ has RDP$_1$. Then also
$P^I$ has RDP$_1$, and  $H_0$ and $H_1$ are normal Riesz $\gamma$-ideals in $P^I$. By
Corollary \ref{co:corol}, they are also normal Riesz ideals in $K$. Clearly, $H_0
\cap H_1=\{ 0\}$. We can deduce that, under the above suppositions, if $K$ has a
smallest nontrivial normal Riesz ideal, then the set $I$ is connected. (Compare with
\cite[Theorem 4.6]{Dvnew}.)
\end{remark}

\end{document}